\documentclass[12pt]{amsart}
\usepackage{amscd}
\usepackage{amssymb}
\usepackage[T1]{fontenc}
\usepackage{eepic}
\usepackage[french]{babel}
\usepackage{times}
\usepackage[latin1]{inputenc}

\allowdisplaybreaks[1]

\newtheorem{thm}{Th\'eor\`eme}[section]
\newtheorem{rema}[thm]{Remarque}
\newtheorem{cor}[thm]{Corollaire}
\newtheorem{lem}[thm]{Lemme}
\newtheorem{prop}[thm]{Proposition}

\newtheorem{defn}[thm]{D\'efinition}
\newtheorem{ex}[thm]{Exemple}
\theoremstyle{definition}

\numberwithin{equation}{section}

\newcommand{\Z}{\mathbb Z}

\newcommand{\n}{\noindent}

\newcommand{\cal}{\mathcal}

\newcommand{\dis}{\displaystyle}

\newcommand{\resumename}{Abstract}
\newenvironment{resume}{\narrower\footnotesize\bf
\noindent\resumename.\quad\footnotesize\rm}{\par\bigskip}

\font\hb=cmbx12
\newcommand\pt{\hbox{\hb .}}
\begin{document}

\title[Algèbre Pré-Gerstenhaber]{ Algèbre Pré-Gerstenhaber à homotopie près}

\author[W. Aloulou, D. Arnal et R. Chatbouri]{Walid Aloulou, Didier Arnal et Ridha Chatbouri}
\date{19/06/2012}
%--------------------------------------------------------------------------------

\begin{abstract}

On \'etudie le concept d'alg\`ebre \`a homotopie pr\`es pour une
structure d\'efinie par deux op\'erations $.$ et $[~,~]$. Un
exemple important d'une telle structure est celui d'alg\`ebre de
Gerstenhaber (commutative et de Lie). La notion d'alg\`ebre de
Gerstenhaber \`a homotopie pr\`es ($G_\infty$ alg\`ebre) est
connue.

Ici nous proposons une d\'efinition d'alg\`ebre pr\'e-Gerstenhaber (pr\'e-commutative et pr\'e-Lie) permettant la construction d'une $\hbox{pre}G_\infty$ alg\`ebre.

Partant d'une structure pr\'e-commutative (Zinbiel) et pr\'e-Lie, on utilise les op\'erades duales correspondantes. Nous donnons la construction explicite de l'alg\`ebre \`a homotopie pr\`es associ\'ee. Celle-ci est une bicog\`ebre (Leibniz et permutative), munie d'une codiff\'erentielle qui est une cod\'erivation des deux coproduits.

\end{abstract}

%-------------------------------------------------------------------------------

\address{Universit\'e de Sousse, Laboratoire de Mathématique Physique Fonctions Spéciales et Applications
et D\'epartement de Math\'ematiques, Institut Pr\'eparatoire aux Etudes
d'Ing\'enieurs de Sfax, Route Menzel Chaker Km 0.5, Sfax, 3018,
Tunisie}
\email{Walid.Aloulou@ipeim.rnu.tn}
\address{
Institut de Math\'ematiques de Bourgogne\\
UMR CNRS 5584\\
Universit\'e de Bourgogne\\
U.F.R. Sciences et Techniques
B.P. 47870\\
F-21078 Dijon Cedex\\France} \email{Didier.Arnal@u-bourgogne.fr}
\address{Universit\'e de Sousse, Laboratoire de Mathématique Physique Fonctions Spéciales et Applications
et D\'epartement de Math\'ematiques\\
Facult\'e des Sciences de Monastir\\
Avenue de l'environnement\\
5019 Monastir\\
Tunisie}
\email{Ridha.Chatbouri@ipeim.rnu.tn}


\thanks{}

%-------------------------------------------------------------------------------
\maketitle

%-------------------------------------------------------------------------------

\begin{resume}

This paper is concerned by the concept of algebra up to homotopy for a structure defined by two operations $.$ and $[~,~]$. An important example of such a structure is the Gerstenhaber algebra (commutatitve and Lie). The notion of Gerstenhaber algebra up to homotopy ($G_\infty$ algebra) is known.

Here, we give a definition of pre-Gerstenhaber algebra (pre-commutative and pre-Lie) allowing the construction of $\hbox{pre}G_\infty$ algebra.

Given a structure of pre-commutative (Zinbiel) and pre-Lie algebra and working over the corresponding dual operads, we will give an explicit construction of the associated pre-Gerstenhaber algebra up to homotopy, this is a bicogebra (Leibniz and permutative) equipped with a codifferential which is a coderivation for the two coproducts.
\end{resume}

\

\section{Introduction}\label{sec1}

%------------------

%\subsection{\bf Introduction.}

\

En 69 \cite{[Q]}, Quillen a montré qu'il existe une dualité entre les structures d'algèbre de Lie et d'algèbre commutative dite dualité de Quillen.

En 94 \cite{[GK]}, Ginzbug et Kapranov ont montré qu'il existe une
dualité entre d'autres types d'algèbres. Ils ont expliqué cette
dualité en introduisant les notions d'opérade $\mathcal{P}$ et
d'opérade duale $\mathcal{P}^!$ qui décrivent et déterminent le
type d'algèbre sur l'opérade $\mathcal{P}$. On dit une
$\mathcal{P}$-algèbre ou une algèbre sur l'opérade $\mathcal{P}$.
Par exemple, les opérades classiques sont les opérades $Lie$,
$Ass$, $Com$. Une $Com$-alg\`ebre (resp. une $Lie$-algèbre) est
une algèbre commutative (resp. de Lie). Les relations de
définition de la structure (multiplication, crochet) sont
quadratiques dans les lois (ainsi la relation d'associativité,
celle de Jacobi font intervenir deux opérations). Les opérades
$Lie$, $Ass$, $Com$ sont quadratiques.

Ginzbug et Kapranov ont montré que toute opérade quadratique admet une opérade duale, ils retrouvent la dualit\'e de Quillen sous la forme $Com^!=Lie$ et $Lie^!=Com$.

Ils ont démontré aussi que  $Ass^!=Ass$.

La structure de $\mathcal{P}$-algèbre libre sur un espace $V$ est donnée par une codifférentielle $Q$ (bilin\'eaire si la loi consid\'er\'ee est binaire) d'une cogèbre libre sur la $\mathcal{P}^!$-cogèbre sur $V$ construite à partir de $V$. La relation $Q^2=0$ est l'\'equation de structure, elle est \'equivalente aux axiomes de d\'efinition de la strucutre d'alg\`ebre.

Dans cette situation, la structure de $\mathcal{P}_\infty$ algèbre (ou $\mathcal P$-alg\`ebre \`a homotopie pr\`es) sur un espace vectoriel $V$ est la donnée d'une codifférentielle $Q$ (non n\'ecessairement bilin\'eaire) sur la $\mathcal{P}^!-$cogèbre sur $V$ (voir \cite{[GK]}).\\

Une alg\`ebre pr\'e-Lie est un espace vectoriel $V$ muni d'une loi
$\diamond$ telle que son antisym\'etris\'ee est une loi
d'alg\`ebre de Lie. Les axiomes de cette structure ont permis \`a
Chapoton et Livernet (\cite{[ChL]}) de r\'ealiser la construction
ci-dessus. Il existe donc une notion d'alg\`ebre pr\'e-Lie \`a
homotopie pr\`es (\cite{[ChL]}).

De m\^eme, il existe une notion d'alg\`ebre pr\'e-commutative, appel\'ee alg\`ebre de Zinbiel (\cite{[Liv]}), ces alg\`ebres sont \'equipp\'ees d'un produit $\wedge$, dont le sym\'etris\'e est associatif et commutatif. Les axiomes permettent de r\'ealiser la construction ci-dessus et de d\'efinir des alg\`ebres de Zinbiel \`a homotopie pr\`es (\cite{[Liv]}).\\

Maintenant, une alg\`ebre de Gerstenhaber est un espace vectoriel $V$ muni de deux lois : un produit commutatif $\wedge$ de  degr\'e 0 et un crochet de Lie $[~,~]$ de degr\'e -1, avec des relations de compatibilit\'es. On peut r\'ealiser la construction ci-dessus pour cette structure (voir \cite{[G]}, et surtout \cite{[BGHHW],[AAC]}). La construction compl\`ete n\'ecessite celle d'une bicog\`ebre $W$ (munie d'un coproduit $\Delta$ et d'un cocrochet $\kappa$ avec des relations de compatibilt\'es) et les deux lois de notre alg\`ebre permettent de construire une seule application $Q$ qui est une cod\'erivation \`a la fois de $\Delta$ et de $\kappa$. Les axiomes d'alg\`ebre de Gerstenhaber sont \'equivalents \`a l'\'equation de structure $Q^2=0$. Ceci permet de d\'efinir les alg\`ebres de Gerstenhaber \`a homotopie pr\`es (\cite{[AAC]}).

Pr\'ecisons cette construction. Comme $(V,\wedge)$ est une alg\`ebre commutative, on construit la cog\`ebre libre associ\'ee $(\mathcal H,\delta)$ et la codiff\'erentielle not\'ee $D$ que d\'efinit $\wedge$. On remarque que le crochet de Lie $[~,~]$ se prolonge en un unique crochet de Lie sur $\mathcal H$ et que les relations de compatibilt\'es entre $\wedge$ et $[~,~]$ sont \'equivalentes \`a : $D$ est une d\'erivation du crochet. On dispose alors d'une alg\`ebre de Lie diff\'erentielle $(\mathcal H,[~,~],D)$. La construction ci-dessus permet de construire la cog\`ebre libre associ\'ee $(W,\Delta)$ et une codiff\'erentielle $Q$, unique prolongement de $D+[~,~]$ \`a $W$. Enfin on prolonge de fa\c con unique $\delta$ en un cocrochet $\kappa$ sur $W$, on obtient la bicog\`ebre codiff\'erentielle $(W,\Delta,\kappa,Q)$.\\

Dans \cite{[Ag]}, Aguiar propose une définition d'une algèbre
pré-Gerstenhaber : un espace $V$ qui possède à la fois une
structure d'algèbre de Zinbiel pour $\wedge$ (de degr\'e 0) et une
structure d'alg\`ebre pré-Lie pour $\diamond$ (de degr\'e -1) avec
des relations de compatibilit\'es.

Reprenons la derni\`ere construction pour une alg\`ebre pr\'e-Gerstenhaber au sens d'Aguiar, on construit
d'une fa\c con unique la cog\`ebre codiff\'erentielle $(\mathcal H,\delta,D)$ d\'efinie pour les alg\`ebres
de Zinbiel par Livernet (\cite{[Liv]}), Aguiar a d\'efini un prolongement naturel du produit pr\'e-Lie $\diamond$
 \`a $\mathcal H$, not\'e $R$. $(\mathcal H,R)$ est bien une alg\`ebre pr\'e-Lie, malheureusement, $D$ n'est pas une d\'erivation de $R$, parce que les relations de compatibilit\'e entre $\wedge$ et $\diamond$ propos\'ees sont trop faibles.\\

On se propose dans ce papier de donner une d\'efinition plus restrictive d'alg\`ebre pr\'e-Gerstenhaber que celle
 d'Aguiar. Plus pr\'ecis\'ement, on imposera des conditions de compatibilit\'es plus fortes.
 On d\'efinit alors un autre prolongement, not\'e $R_2$ de $\diamond$ \`a $\mathcal H$ tel que $(\mathcal H,R_2,D)$ est
  une alg\`ebre pr\'e-Lie diff\'erentielle. On pourra alors achever la construction de la bicog\`ebre codiff\'erentielle
   $(W,\Delta,\kappa,Q)$ dans le cadre des alg\`ebres pr\'e-Gerstenhaber. On a ainsi une construction explicite de
    l'algèbre pr\'e-Gerstenhaber à homotopie près associ\'ee.\\

%----------

\section{Algèbres à homotopie près ou $\mathcal{P}_\infty$ algèbres}

\

\subsection{G\'en\'eralit\'es}

\

Soit $V=\bigoplus_{n\in\Z}V_n$ un espace vectoriel $\mathbb{Z}$ gradué. Le degr\'e d'un élément homogène $x$ dans
$V$ est noté $|x|$. On notera $T^+(V)$ l'espace $\bigoplus_{n>0}\bigotimes^n V$, gradu\'e par $|x_1\otimes\dots\otimes x_n|=|x_1|+\dots+|x_n|$.\\

\begin{defn}

\

\noindent
1) Une algèbre graduée est un espace vectoriel gradué $V$ muni d'une application bilinéaire
$ b: V \otimes V \rightarrow V$ de degré 0 ($|b|=0$)
c'est \`a dire :
$$b(V_i\otimes V_j)\subset V_{i+j}.$$

%\noindent
%2) Un morphisme \ ${\cal F} : (V, b)
%\longrightarrow (V', b')$ d'algèbres est une
%application de degré 0 vérifiant: $$ b'\circ({\cal F \otimes \cal F})  =  {\cal F}\circ b.$$

\noindent
2) Une dérivation $d$ de l'algèbre $(V, b)$ est une application vérifiant: $$d \circ b = b\circ(d \otimes id + id \otimes d).$$ Si $d$ est de degré $1$ et $d^2=0$, on dit que $d$ est une différentielle de $(V,b)$.\\
\end{defn}

\newpage

\begin{defn}

\

\noindent 1) Une cogèbre graduée est un espace vectoriel gradué
$\cal C$ muni d'une application linéaire $ \Delta: \cal
C\longrightarrow \cal C \otimes \cal C$ dite coproduit de $\cal C$
vérifiant:
$$\Delta {\cal C}_k \subset \dis\sum_{i+j =k} {\cal C}_i\otimes {\cal C}_j$$

%\noindent
%2) Un morphisme \ ${\cal F} : (\mathcal{C}, \Delta)\longrightarrow (\mathcal{C}', \Delta')$ de cogèbres est une application vérifiant:
%$$
%({\cal F \otimes \cal F})\circ \Delta = \Delta' \circ {\cal F}.
%$$

\noindent
2) Une codérivation $Q$ de la cogèbre $(\mathcal{C}, \Delta)$ est une application vérifiant:
$$
\Delta \circ Q = (Q \otimes id + id \otimes Q)\circ \Delta.
$$
Si $Q$ est de degré $1$ et $Q^2=0$, on dit que $Q$ est une codifférentielle de $(\mathcal{C},\Delta)$. Dans ce cas, on dit que $(\mathcal C,\Delta,Q)$ est une cog\`ebre codiff\'erentielle.

%\noindent
%4) Un morphisme de cog\`ebres codiff\'erentielles ${\cal F} : (\mathcal{C}, \Delta,Q)\longrightarrow (\mathcal{C}', \Delta', Q')$ est un morphisme de
%cog\`ebres v\'erifiant $\mathcal{F}\circ Q=Q'\circ\mathcal{F}$.\\
\end{defn}

Par définition, l'espace $V[1]$ est le m\^eme espace que $V$, mais avec un d\'ecalage du degr\'e : le degré d'un élément homogène $x$ dans $V[1]$ noté $deg(x)$ devient $deg(x)=|x|-1$.
Rappelons qu'on peut faire correspondre \`a toute application
$n$-lin\'eaire $|~|$-antisym\'etrique (resp. $|~|$-sym\'etrique) $\phi:V\otimes\dots\otimes V\rightarrow V$ une application $n$-lin\'eaire $deg$-sym\'etrique (resp. $deg$-antisym\'etrique) $\phi':V[1]\otimes\dots\otimes V[1]\rightarrow V[1]$ en posant
$$
\phi'(x_1,\dots,x_n)=(-1)^{\sum_{i=1}^n(n-i)deg(x_i)}\phi(x_1,\dots,x_n)
$$
(voir \cite{[AAC1]}).\\

Les bonnes structures algèbriques sont des lois associées à une opérade quadratique $\mathcal P$ (\cite{[GK]}). En effet, si $V$ est un espace vectoriel gradué, on peut dans ce cas construire la cogèbre colibre $(W,\Delta)$ sur l'opérade duale $\mathcal P^!$ engendrée par le décalé $V[1]$ de $V$.

Dans cette situation, dire qu'une loi $b:V\otimes V\rightarrow V$ de degré 0 est une structure de type $\mathcal P$, c'est dire que sa décalée $b':V[1]\otimes V[1]\rightarrow V[1]$ est bilin\'eaire, de degré 1 et vérifie les symétries associées à l'opérade $\mathcal P^!$, donc est prolongeable de façon unique en une codérivation $Q$ de $(W,\Delta)$ et la loi $b$ vérifie les axiomes de la structure si et seulement si, $Q$ vérifie l'équation de structure $[Q,Q]=2Q^2=0$.\\

Toujours dans ce cas, on parlera de $\mathcal P_\infty$ alg\`ebre ou d'alg\`ebre \`a homotopie pr\`es pour toute cog\`ebre codiff\'erentielle $(W,\Delta,Q)$ correspondante \`a $(W,\Delta)$, $Q$ quadratique ou non.\\

\begin{defn}{\rm \cite{[GK]}}

\

Une structure de $\mathcal{P}_\infty$ algèbre sur un espace vectoriel $V$ est définie par la donnée d'une codifférentielle $Q$ sur la $\mathcal{P}^!$-cogèbre $(W,\Delta)$ construite \`a partir de $V$.\\
\end{defn}

En particulier, si $(V,b,d)$ est une alg\`ebre diff\'erentielle, on peut prolonger $d+b'$ en une unique cod\'erivation $Q$ de degr\'e 1 de $(W,\Delta)$ telle que $Q^2=0$, c'est \`a dire $(W,\Delta,Q)$ est une $\mathcal P_\infty$ alg\`ebre. D\'ecrivons quelques exemples.\\

%---------------------------------------------

\subsection{Algèbres de Lie et alg\`ebres commutatives à homotopie près ($L_\infty$ et $C_\infty$ alg\`ebres)}

\

Les structures d'alg\`ebre de Lie et d'alg\`ebre commutative sont associ\'ees \`a des op\'erades quadratiques : les op\'erades $Lie$ et $Com$. De plus, Ginzburg et Kapranov ont montr\'e que $Lie^!=Com$ et $Com^!=Lie$. On peut donc d\'efinir les notions de $L_\infty$ et $C_\infty$ alg\`ebres.\\

\subsubsection{Alg\`ebres de Lie et $L_\infty$ alg\`ebres}

\

\begin{defn}

\

Soit $\mathfrak{g}$ un espace vectoriel gradué. Une structure d'alg\`ebre de Lie gradu\'ee sur $\mathfrak g$ est la donn\'ee d'un crochet bilin\'eaire $[~,~]$ tel que :
\begin{itemize}
\item $[~,~]$ est de degr\'e 0:
$$
\forall i,j,~~[\mathfrak{g}_i,\mathfrak{ g}_j]\subset\mathfrak{ g}_{i+j},
$$

\item $[~,~]$ est antisym\'etrique :
$$\forall x,y\in\mathfrak{g}, ~~[x,y]=-(-1)^{|x||y|}[y,x],
$$

\item $[~,~]$ v\'erifie l'identit\'e de Jacobi :
$$
\forall x,y,z\in\mathfrak{g}, ~~(-1)^{|x||z|}[[x,y],z]+(-1)^{|y||x|}[[y,z],x]+(-1)^{|z||y|}[[z,x],y]=0.
$$
\end{itemize}

Si de plus, il existe une diff\'erentielle $d:\mathfrak{ g}\longrightarrow\mathfrak{ g}$ ({\sl i.e.} $d^2=0$), telle que $|d|=1$ et
$$
d([x,y])=[dx,y]+(-1)^{1.|x|}[x,dy].
$$
On dit que $(\mathfrak{g},[~,~],d)$ est une alg\`ebre de Lie diff\'erentielle gradu\'ee.\\
\end{defn}

Rappelons qu'on a pos\'e $deg(x)=|x|-1$. on notera aussi simplement $x$ ce degr\'e. On note aussi
$$
\varepsilon_x\left(\begin{smallmatrix}x_1&\dots&x_n\\ x_{i_1}&\dots&x_{i_n}\end{smallmatrix}\right)=\varepsilon_x(\sigma)
$$
la signature de la permutation $\sigma=\left(\begin{smallmatrix}1&\dots&n\\ i_1&\dots&i_n\end{smallmatrix}\right)$, en tenant compte des degr\'es de $x_j$, autrement dit, $\varepsilon_x$ est l'unique morphisme de $S_n$ dans $\mathbb R$ tel que $\varepsilon_x((i,j))=(-1)^{x_ix_j}$.\\

Puisque l'opérade duale de $Lie$ est $Com$, une $Lie^!$-cog\`ebre est une cog\`ebre coassociative et cocommutative.\\

\begin{prop}

\

Soit $\mathfrak{g}$ un espace gradué, notons $S^+(\mathfrak g[1])$ l'espace $^{T^+(\mathfrak g[1])}\diagup_{\langle x\otimes y-(-1)^{xy}y\otimes x\rangle}$, alors la $Lie^!$-cog\`ebre colibre engendr\'ee par $\mathfrak g[1]$ est
$$
\Big(S^+(\mathfrak{g}[1]),\Delta\Big)
$$
avec, si $I=\{i_1<\dots<i_k\}$, $x_I$ d\'esigne $x_{i_1}\dots x_{i_k}$ et :
$$
\Delta(x_1\dots x_n)=\displaystyle\sum_{\begin{smallmatrix}I\cup J=\{1,\dots n\}\\ \#I,\#J>0\end{smallmatrix}}\varepsilon_{x}\left(\begin{smallmatrix}x_1&\dots&x_n\\ x_I&&x_J\end{smallmatrix}\right)x_I\otimes x_J.
$$
\end{prop}
On v\'erifie directement (voir par exemple \cite{[AMM]}) que cette cog\`ebre est coassociative :
$$
(\Delta\otimes id)\circ\Delta=(id\otimes\Delta)\circ\Delta,
$$
et cocommutative : si $\tau$ est la volte $\tau(x\otimes y)=(-1)^{xy}y\otimes x$, alors
$$
\tau\circ\Delta=\Delta.
$$

Une structure de $L_\infty$ alg\`ebre sur $\mathfrak g$ est un triplet $\Big(S^+(\mathfrak{g}[1]),\Delta, Q\Big)$, formé de cette cog\`ebre cocommutative coassociative munie d'une cod\'erivation $Q$ de $\Delta$ de degr\'e $1$ et de carré nul. En particulier\\

\begin{prop}{\rm \cite{[AAC]}}

\

Si $(\mathfrak{g}, [~,~],d)$ est une algèbre de Lie diff\'erentielle gradu\'ee, on pose :
$$
Q_1(x)=dx,~Q_2(x.y)=(-1)^x[x,y], ~Q_k=0, \forall k\geq3,
$$
et
$$
Q(x_1\dots x_n)=\sum_{\begin{smallmatrix}I\cup J=\{1,\ldots,n\}\\ I\neq\emptyset\end{smallmatrix}}\varepsilon_x\left(\begin{smallmatrix}x_1\dots x_n\\x_Ix_J\end{smallmatrix}\right)Q_{\#I} (x_{I}).x_J.
$$

Alors $\left(S^+(\mathfrak{g}[1]),\Delta,Q\right)$ est une $L_\infty$ alg\`ebre dite la $L_\infty$ alg\`ebre enveloppante
de $\mathfrak{g}$.\\
\end{prop}

%--------------

\subsubsection{Algèbre commutative à homotopie près ($C_\infty$ alg\`ebre)}

\

\begin{defn}

\

On dit que $(A,\pt, d)$ est une algèbre commutative différentielle graduée si :
\begin{itemize}
\item $\pt$ est une loi de degr\'e 0 et commutative:
$$
\forall x,y\in A, ~x\pt y=(-1)^{|x||y|}y\pt x,
$$

\item $\pt$ est associatif :
$$
\forall x,y,z\in A, ~(x\pt y)\pt z=x\pt (y\pt z),
$$

\item $d$ est une diff\'erentielle de degr\'e 1, de carr\'e nul, et telle que :
$$
\forall x,y\in A, ~d(x\pt y)=dx\pt y+(-1)^{|x|}x\pt dy.
$$
\end{itemize}
\end{defn}

Comme l'opérade duale de $Com$ est $Lie$, on construit la
$Com^!$-cog\`ebre colibre ainsi :

\begin{itemize}
\item On considère $A[1]$ muni du degr\'e $deg(x)=|x|-1=x$

\item Une permutation $\sigma\in S_{p+q}$ ($p,q\geq1$) est dite un $(p,q)$-shuffle si elle vérifie :
$$
\sigma(1)<\cdots<\sigma(p) \hskip0.4cm \hbox{et} \hskip0.4cm \sigma(p+1)<\cdots<\sigma(p+q).
$$
On note $Sh(p,q)$ l'ensemble des $(p,q)$-shuffles.

\item On d\'efinit ensuite le produit shuffle sur $\bigotimes^+ A[1]$ par :

$$
sh_{p,q}\left(x_1\otimes...\otimes x_p , x_{p+1}\otimes...\otimes x_{p+q}\right)=\sum_{\sigma\in Sh(p,q)} \varepsilon_{x}(\sigma^{-1}) x_{\sigma^{-1}(1)}\otimes...\otimes x_{\sigma^{-1}(p+q)}.
$$

\item On d\'efinit alors l'espace quotient
$$
\mathcal H=\underline{\bigotimes}^+A[1]=\bigoplus_{n\geq1}~^{\bigotimes^nA[1]}\diagup_{\sum_{p+q=n}Im(sh_{p,q})}.
$$
\end{itemize}

\begin{prop}

\

La $Com^!$-cog\`ebre colibre associ\'ee \`a $A$ est l'espace $\mathcal H$ muni du cocrochet $\delta$ d\'efini par :
$$
\delta(x_1\underline{\otimes}...\underline{\otimes}x_n)= \displaystyle\sum_{j=1}^{n-1}
x_1\underline{\otimes}...\underline{\otimes} x_j\bigotimes x_{j+1}\underline{\otimes}...\underline{\otimes} x_n-\varepsilon_x(\tau)
x_{j+1}\underline{\otimes}...\underline{\otimes} x_n \displaystyle\bigotimes x_1\underline{\otimes}...\underline{\otimes} x_j.
$$
\end{prop}

Le coproduit $\delta$ est coantisym\'etrique :
$$
\tau\circ\delta=-\delta
$$
et v\'erifie l'identit\'e de coJacobi (ici $\tau_{23}=id\otimes\tau$ et $\tau_{12}=\tau\otimes id$) :
$$
\Big(id^{\otimes3}+\tau_{12}\circ\tau_{23}+\tau_{23}\circ\tau_{12}\Big)\circ(\delta\otimes id)\circ\delta=0.
$$
On dira que $\left(\mathcal H,\delta\right)$ est une cog\`ebre de Lie de cocrochet $\delta$.\\

Une structure de $C_\infty$ alg\`ebre sur $A$ est la cog\`ebre de Lie $\Big(\mathcal H,\delta, Q\Big)$ munie d'une cod\'erivation $Q$ de $\delta$ de degr\'e $1$ et de carré nul. En particulier\\

\begin{prop}{\rm \cite{[AAC]}}

\

Si $(A, \pt,d)$ est une algèbre commutative diff\'erentielle
gradu\'ee, on pose :
$$
Q_1(x)=dx,~Q_2(x\underline\otimes y)=(-1)^xx\pt y, ~Q_k=0, \forall k\geq3,
$$
et
$$
Q(x_1\underline{\otimes}...\underline{\otimes}x_n)=\displaystyle\sum_{\begin{smallmatrix}1\leq r\leq n\\ 0\leq j\leq n-r\end{smallmatrix}}(-1)^{\sum_{i\leq j}x_i} x_1\underline{\otimes}...\underline{\otimes} x_j\underline\otimes Q_r(x_{j+1}\underline{\otimes}...\underline{\otimes}
x_{j+r})\underline\otimes x_{j+r+1}\underline{\otimes}...\underline{\otimes} x_n.
$$

Alors $\left(\mathcal H,\delta,Q\right)$ est une $C_\infty$ alg\`ebre dite la $C_\infty$ alg\`ebre enveloppante de $A$.\\
\end{prop}

%----------------------------------------------------------------

\section{Algèbre de Gerstenhaber à homotopie près ($G_\infty$ alg\`ebre)}

\

Une alg\`ebre de Gerstenhaber est un espace vectoriel gradu\'e, muni de deux lois de degr\'es diff\'erents et de sym\'etries oppos\'ees. Voir \cite{[G],[A]}.\\

\begin{defn}

\

On dit que $(\mathcal{G},\wedge, [~,~],d)$ est une alg\`ebre de Gerstenhaber diff\'erentielle graduée si :
\begin{itemize}
\item $(\mathcal{G},\wedge,d)$ est une algèbre commutative diff\'erentielle graduée, $|\wedge|=0$,

\item $(\mathcal{G}[1],[~,~],d)$ est une algèbre de Lie diff\'erentielle gradu\'ee, ~~$|[~,~]|=-1$,

\item  La relation de compatibilité suivante (relation de Leibniz) entre $\wedge$ et $[~,~]$ est v\'erifi\'ee :
$$
[\alpha,\beta\wedge\gamma]=[\alpha,\beta]\wedge\gamma+(-1)^{|\beta|(|\alpha|-1)}\beta\wedge[\alpha,\gamma].
$$
\end{itemize}
\end{defn}

\vskip0.2cm

Le prototype des algèbres de Gerstenhaber est l'espace $T_{poly}(\mathbb{R}^d)$ des multichamps de vecteurs muni du crochet de Schouten $[~,~]$ et du
produit extérieur $\wedge$.\\

Pour d\'efinir la notion d'alg\`ebre de Gerstenhaber \`a homotopie pr\`es, on proc\`ede en trois \'etapes.

On construit d'abord la cog\`ebre de Lie $(\mathcal H,\delta)$ associ\'ee \`a $(\mathcal G,\wedge,d)$ comme ci-dessus. On prolonge \`a $\mathcal H$ l'application $d+\wedge$ en une cod\'erivation de $\delta$ de carr\'e nul, que l'on note maintenant $D$.

On prolonge ensuite le crochet $[~,~]$ de $\mathcal G[1]$ \`a $\mathcal H$ en posant (voir \cite{[F],[AAC]}), si $X=\alpha_1\underline{\otimes}...\underline{\otimes}\alpha_p$, $Y=\alpha_{p+1}\underline{\otimes}...\underline{\otimes}\alpha_{p+q}$ :
$$
[X,Y]=\hskip-0.2cm\sum_{\begin{smallmatrix}\sigma\in Sh(p,q)\\
k,\sigma^{-1}(k)\leq p<\sigma^{-1}(k+1)\end{smallmatrix}}\hskip-0.1cm
\varepsilon_\alpha(\sigma^{-1})\alpha_{\sigma^{-1}(1)}\underline{\otimes}\dots\underline{\otimes}[\alpha_{\sigma^{-1}(k)},\alpha_{\sigma^{-1}(k+1)}]\underline{\otimes}\dots\underline{\otimes}\alpha_{\sigma^{-1}(p+q)}.
$$
On note $deg(X)=deg(\alpha_1\underline{\otimes}\dots\underline{\otimes}\alpha_p)=\alpha_1+\cdots+\alpha_p=x$. Alors :\\

\begin{prop}

\

$(\mathcal{H},[~,~],D)$ est une alg\`ebre de Lie diff\'erentielle gradu\'ee.\\
\end{prop}

On peut ensuite construire la $L_\infty$ algèbre enveloppante $\left(S^+(\mathcal{H}[1]),\Delta,Q\right)$ de $\mathcal{H}$.

On consid\`ere l'espace $\mathcal{H}[1]$, muni du degré $deg'(X)=deg(X)-1=x'$. On pose $\ell_2(X.Y)=(-1)^{x'}[X,Y]$, et on prolonge $D$ et $\ell_2$ en $m$ et $\ell$ \`a $S^+(\mathcal{H}[1])$ par :
$$
m(X_1\dots X_n)=\sum_{j=1}^{n}\varepsilon_{x'}\left(\begin{smallmatrix}x_1\dots x_n\\
x_j,x_1,\dots\widehat{j}\dots x_n\end{smallmatrix}\right)D(X_j).X_1\dots\widehat{j}\dots X_n,
$$
et
$$
\ell(X_1\dots X_n)=\sum_{i<j}\varepsilon_{x'}\left(\begin{smallmatrix}x_1\dots x_n\\
x_i,x_j,x_1,\dots\widehat{i,j}\dots x_n\end{smallmatrix}\right)\ell_{2}(X_i.X_j).X_1...\widehat{ij}...X_n.
$$

On obtient une cod\'erivation $Q=(m+\ell)$ de $\Delta$ vérifiant $deg'(Q)=1$ et $Q^2=0$.

Enfin (voir \cite{[BGHHW],[AAC]}), le cocrochet $\delta$ de la $C_\infty$ alg\`ebre $(\mathcal{H},\delta,D)$ se prolonge en l'application $\kappa$, d\'efinie sur
$S^+(\mathcal{H}[1])$, par :
\begin{align*}
&\kappa(X_1\dots X_n)=\sum_{\begin{smallmatrix}1\leq s\leq n\\
I\cup J=\{1,\dots,n\}\setminus\{s\}\end{smallmatrix}}
(-1)^{\sum_{i<s}x_i'}\sum_{\begin{smallmatrix}U_s\underline{\otimes}
V_s=X_s\\ U_s,V_s\neq\emptyset\end{smallmatrix}
}(-1)^{u_s'}\times\cr&\times\left(\varepsilon_{x'}\left(\begin{smallmatrix}
x_1\dots x_n\\ x_I~u_s~v_s~x_J\end{smallmatrix}\right)X_I.
U_s\bigotimes V_s . X_J+\varepsilon_{x'}\left(\begin{smallmatrix}
x_1\dots x_n\\ x_I~v_s~u_s~x_J\end{smallmatrix}\right)X_I.
V_s\bigotimes U_s . X_J\right),
\end{align*}
avec \begin{align*}\varepsilon_{x'}\left(\begin{smallmatrix}
x_1\dots x_n\\ x_I~u_s~v_s~x_J\end{smallmatrix}\right)=
\varepsilon_{x'}\left(\begin{smallmatrix} x_1\dots x_n\\
x_I~x_s~x_J\end{smallmatrix}\right)(-1)^{\sum_{\begin{smallmatrix}i<s\\ i\in
J\end{smallmatrix}}x_i'}(-1)^{\sum_{\begin{smallmatrix}i>s\\ i\in I\end{smallmatrix}}x_i'}.
\end{align*}
$\kappa$ est cosym\'etrique et v\'erifie les relations de coJacobi
et coLeibniz avec $\Delta$ :
$$
(id\otimes\Delta)\circ\kappa=(\kappa\otimes
id)\circ\Delta+\tau_{12}'\circ(id\otimes\kappa)\circ\Delta
$$
ou
$$
(\Delta\otimes id)\circ\kappa=(id\otimes
\kappa)\circ\Delta+\tau_{23}'\circ(\kappa\otimes id
)\circ\Delta.
$$
De plus, l'op\'erateur $Q=m+\ell$ est une cod\'erivation de $\kappa$.\\

\begin{defn}{\rm \cite{[AAC]}}

\

Une structure de $G_\infty$ alg\`ebre sur $\mathcal G$ est la donn\'ee d'une codiff\'erentielle $Q$ de la bicog\`ebre $\left(S^+(\underline{\displaystyle\bigotimes}^+(\mathcal{G}[1])[1]),\Delta,\kappa\right)$ telle que $Q$ est une cod\'erivation de $\Delta$ et de $\kappa$, de degr\'e $1$ et de carr\'e nul.\\
\end{defn}

En particulier, si $(\mathcal G,\wedge,[~,~],d)$ est une algèbre de Gerstenhaber différentielle, alors $\left(S^+(\underline{\displaystyle\bigotimes}^+(\mathcal{G}[1])[1]),\Delta,\kappa,Q=m+\ell\right)$ est une $G_\infty$ algèbre dite la $G_\infty$ algèbre enveloppante de $\mathcal G$.\\

%---------------------------------
\section{Algèbre pré-Lie et pr\'e-commutative à homotopie près}

\

\subsection{Algèbre pré-Lie à homotopie près ($preL_\infty$ alg\`ebre)}

\

La notion d'alg\`ebre pr\'e-Lie a \'et\'e \'etudi\'ee par Livernet et Chapoton (\cite{[Liv],[ChL]}). Une loi pr\'e-Lie est une loi binaire dont l'antisym\'etris\'e est un crochet de Lie. Plus pr\'ecis\'ement :\\

\begin{defn}

\

Une alg\`ebre pré-Lie (\`a droite) gradu\'ee $(V,\diamond)$ est un espace gradu\'e $V$ muni d'un produit $\diamond$ de degré $0$
vérifiant :
$$
\forall x,y,z\in V,~~(x\diamond y)\diamond z-x\diamond (y\diamond z)=(-1)^{|y||z|}\big((x\diamond z)\diamond y -x\diamond(z\diamond y)\big).
$$

Si de plus, $d:V\longrightarrow V$ est une diff\'erentielle de degré $1$ telle que
$$
d(x\diamond y)=dx\diamond y+(-1)^{1.|x|}x\diamond dy,
$$
on dira que $(V,\diamond,d)$ est une alg\`ebre pré-Lie diff\'erentielle gradu\'ee.\\
\end{defn}

Cette structure est associ\'ee \`a une op\'erade quadratique, l'op\'erade $preLie$. L'op\'erade duale, d\'etermin\'ee par \cite{[ChL]}, est l'op\'erade permutative : $preLie^!=Perm$.

Rappelons qu'une alg\`ebre permutative (\`a droite) $(V,.)$ est un espace gradu\'e $V$ muni d'un produit . de degré $0$, vérifiant :
$$
\forall x,y,z\in V,~~x. (y. z)=(-1)^{|y||z|}x.(z.y).
$$

\begin{defn}

\

Une $preLie^!$-cog\`ebre (ou cogèbre permutative) est un espace vectoriel gradué $\mathcal{C}$ muni d'une comultiplication $\Delta:\mathcal{C}\longrightarrow\mathcal{C}\otimes\mathcal{C}$ de degré $0$ vérifiant :
$$
(id\otimes \Delta)\circ \Delta= \tau_{23}\circ(id\otimes \Delta)\circ \Delta.
$$
\end{defn}

\begin{prop} {\rm \cite{[ChL]}}

\

Soit $V$ un espace vectoriel gradué. Alors la cog\`ebre permutative colibre associ\'ee \`a $V[1]$ est $\Big(V[1]\otimes S(V[1]),\Delta\Big)$ o\`u $\Delta$ est d\'efini par $\Delta (x\otimes 1)=0$ et :
$$
\Delta(x_0\otimes x_1\dots x_n)= \displaystyle\sum_{\begin{smallmatrix}
0\leq k\leq n-1 \\ \sigma\in Sh_{k,1,n-k-1}\end{smallmatrix}}\varepsilon_x(\sigma)x_0\otimes(x_{\sigma(1)}\dots
x_{\sigma(k)})\bigotimes x_{\sigma(k+1)}\otimes(x_{\sigma(k+2)}\dots x_{\sigma(n)}).
$$
(Ici, $Sh_{k,1,n-k-1}$ est l'ensemble des permutations $\sigma$ de $S_n$ telles que $\sigma(1)<\dots<\sigma(k)$ et $\sigma(k+2)<\dots<\sigma(n)$).\\
\end{prop}

\begin{rema}

\

Identifions $S^{n+1}(V[1])$ avec un sous espace de $V[1]\otimes S^n(V[1])$ en posant
$$
x_0\dots x_n=\sum_{\sigma\in S_{n+1}} \varepsilon_x(\sigma^{-1})x_{\sigma(0)}\otimes x_{\sigma(1)}\dots x_{\sigma(n)}
$$
On a alors :
\begin{align*}
\Delta (x_0\otimes x_1\dots x_n)&=\sum_{\begin{smallmatrix}
I\cup J=\{1,\dots,n\}\\ J\neq\emptyset
\end{smallmatrix}}\varepsilon_{x}(\begin{smallmatrix} x_1&\dots&x_n\\ x_I&&x_J\end{smallmatrix})(x_0\otimes x_I)\bigotimes x_J\cr
&=x_0\bigotimes x_1\dots x_n+(-1)^{x_0}x_0\otimes\Delta'(x_1\dots x_n)
\end{align*}
où $\Delta'(x_1\dots x_n)$ est le coproduit de la cogèbre cocommutative colibre $S^{+}(V[1])$.

Il est alors clair que $\Delta(x_0\dots x_n)=\Delta'(x_0\dots x_n)$.\\
\end{rema}

\begin{defn}

\

Une $preL_\infty$ alg\`ebre est une cog\`ebre permutative codifférentielle $\Big(V[1]\otimes S(V[1]),\Delta, Q\Big)$ telle que $Q$ est une cod\'erivation de $\Delta$ de degr\'e $1$ et $Q^2=0$.\\
\end{defn}

\begin{prop} {\rm \cite{[ChL]}}

\

Si \ $(V, \diamond,d)$ est une algèbre pré-Lie
diff\'erentielle gradu\'ee. On pose \vskip0.15cm
$$
Q_1(x)=dx,~~ Q_2(x_1\otimes x_2)=(-1)^{x_1}x_1\diamond x_2,~~Q_k=0, \forall k\geq3
$$
et
$$\aligned
Q(x_0\otimes x_1\dots x_n)&=Q_1(x_0)\otimes x_1\dots x_n+(-1)^{x_0}\sum_{k=1}^{n}(-1)^{ \sum_{i<k}x_i}x_0\otimes x_1\dots Q_1(x_k)\dots x_n+\\
&\hskip 1cm+\sum_{\sigma\in
Sh_{1,n-1}}\varepsilon_x(\sigma)Q_2(x_0\otimes x_{\sigma(1)})\otimes x_{\sigma(2)}\dots x_{\sigma(n)}+\\
&\hskip 1cm+(-1)^{x_0}\hskip-0.5cm\sum_{\sigma\in Sh_{1,1,n-2}}\varepsilon_x(\sigma) x_0\otimes Q_2(x_{\sigma(1)}\otimes x_{\sigma(2)}).x_{\sigma(3)}\dots x_{\sigma(n)}.
\endaligned
$$

Alors $\Big(V[1]\otimes S(V[1]),\Delta,Q\Big)$ est une $preL_\infty$ alg\`ebre dite la $preL_\infty$ alg\`ebre enveloppante de $(V,\diamond,d)$.\\
\end{prop}

%--------------

\subsection{Algèbre  pr\'e-commutative (de Zinbiel) à homotopie près ($Z_\infty$ alg\`ebre)}

\

Une loi d'alg\`ebre pr\'e-commutative, ou d'alg\`ebre de Zinbiel, est une loi binaire dont le sym\'etris\'e est une loi commutative et associative. Plus pr\'ecis\'ement :\\

\begin{defn} {\rm \cite{[L1],[Liv]}}

\

On dit que $(V,\wedge,d)$ est une algèbre de Zinbiel (ou
pr\'e-commutative) \`a droite, différentielle et graduée si $V$
est un espace gradu\'e muni d'un produit $\wedge$ de degré $0$ et
d'une différentielle $d$ de degré $1$ vérifiant :
\begin{itemize}
\item $\forall x,y,z\in V$,~~$(x\wedge y)\wedge z=x\wedge (y\wedge z)+(-1)^{|y||z|}x\wedge (z\wedge y)$,
\item $\forall x,y\in V$,~~$d(x\wedge y)=dx\wedge  y+(-1)^{|x|}x\wedge dy$.\\
\end{itemize}
\end{defn}

Cette structure est associ\'ee \`a une op\'erade quadratique, l'op\'erade $Zinb$. L'op\'erade duale, d\'etermin\'ee par \cite{[L1],[Liv]}, est l'op\'erade Leibniz : $Zinb^!=Leib$.

Rappelons qu'une alg\`ebre de Leibniz (\`a droite) $(V,[~,~])$ est un espace gradu\'e $V$ muni d'un crochet $[~,~]$ de degré $0$ vérifiant :
$$
\forall x,y,z\in V,~~[[x, y],z]=[x,[y, z]]+(-1)^{|y||z|}[[x, z],y].
$$

\begin{defn} {\rm \cite{[Liv]}}

\

Une $Zinb^!$-cog\`ebre ou cogèbre de Leibniz est un espace vectoriel gradué $\mathcal{C}$ muni d'une comultiplication $\delta:\mathcal{C}\longrightarrow\mathcal{C}\otimes\mathcal{C}$ de degré $0$ vérifiant :
$$
(id\otimes \delta)\circ \delta=(\delta\otimes id-\tau_{23}\circ(\delta\otimes id))\circ \delta.
$$
\end{defn}

La cog\`ebre de Leibniz colibre engendr\'ee par $V[1]$ est donn\'ee par :\\

\begin{prop} \cite{[Liv]}

\

Soit $V$ un espace vectoriel gradué. Alors la cogèbre de Leibniz colibre graduée engendr\'ee par $V[1]$ est $(T^+(V[1]),\delta)$ où $\delta$ est d\'efini par :
$$
\delta(x_1\otimes \dots \otimes x_n)=\sum_{1\leq k\leq n-1}(x_1\otimes \dots \otimes x_{k})\bigotimes\mu_{n-k}(x_{k+1}\otimes \dots \otimes x_{n}),
$$
les $\mu_j$ sont d\'efinis par r\'ecurrence ainsi : $\mu_1=id$, et, si $\tau_n$ est le cycle $(1,\dots,n)$ de $S_n$,
$$
\mu_{n+1}=\mu_n\otimes id-(\mu_n\otimes id)\circ\tau_{n+1}^{-1}.
$$
(Comme pour la volte, l'action des $\mu_j$ sur les produits tensoriels est sign\'ee).\\
\end{prop}

On peut montrer par r\'ecurrence :

\begin{lem}

\

Pour tout $p$, $q$ positif,
$$
\mu_{p+q}\circ sh_{p,q}=0.
$$
\end{lem}

\begin{defn}

\

Une structure de $Z_\infty$ alg\`ebre est la donn\'ee d'une cog\`ebre de Leibniz codifférentielle $\Big(T^+(V[1]),\delta, Q\Big)$ telle que $Q$
est une cod\'erivation de $\delta$ de degr\'e $1$ et de carr\'e nul.\\
\end{defn}

\begin{prop} \cite{[Liv]}

\

Soit $(V,\wedge,d)$ une algèbre de Zinbiel diff\'erentielle gradu\'ee. On pose
$$
Q_1(x)=dx,~Q_2(x\otimes y)=(-1)^{x}x\wedge y,~ Q_k=0,~\forall k\geq3,
$$
et
$$\aligned
Q(x_1\otimes \dots \otimes x_n)&=\sum_{k=1}^{n}(-1)^{ \sum_{i<k}x_i}x_1\otimes...\otimes
x_{k-1}\otimes Q_1(x_k)\otimes x_{k+1}\otimes\dots\otimes x_n +\\
&+Q_2(x_1\otimes x_2)\otimes x_3\otimes\dots \otimes x_n+\\
&+\sum_{k=2}^{n-1}(-1)^{\sum_{i<k}x_i}x_1\otimes\dots\otimes Q_2\circ
\mu_2(x_k\otimes x_{k+1})\otimes\dots\otimes x_n.
\endaligned
$$
Alors $\Big(T^+(V[1]),\delta, Q\Big)$ est une $Z_\infty$
alg\`ebre appel\'ee la $Z_\infty$ alg\`ebre enveloppante de $(V,\wedge,d)$.\\
\end{prop}

%---------------------------------------------------

%\section{Algèbre pré-Gerstenhaber et alg\`ebre pr\'e-Gerstenhaber à homotopie près ($preG_\infty$ alg\`ebre) }

%\

\section{Algèbre pré-Gerstenhaber}

\

Une alg\`ebre pr\'e-Gerstenhaber est une bialg\`ebre gradu\'ee
$\mathcal G$, pour une loi $\wedge$ pr\'e-commutative, de degr\'e
0 sur $\mathcal G$, et une loi $\diamond$ pr\'e-Lie sur $\mathcal
G[1]$ (donc de degr\'e -1 sur $\mathcal G$), avec des relations de
compatibilit\'es.

Lorsqu'on sym\'etrise $\wedge$ et antisym\'etrise $\diamond$, on obtient une structure d'alg\`ebre de Gerstenhaber sur $\mathcal G$.\\

Plus pr\'ecis\'ement, nous proposons ici la d\'efinition suivante :

\begin{defn}

\

On dit que $(\mathcal{G},\wedge,\diamond)$ est une alg\`ebre pré-Gerstenhaber \`a droite graduée si :
\begin{itemize}
\item $(\mathcal{G},\wedge)$ est une algèbre de Zinbiel \`a droite graduée, $|\wedge|=0$.

\item $(\mathcal{G}[1],\diamond)$ est une algèbre pré-Lie \`a droite gradu\'ee, $|\diamond|=-1$.
\item  On impose les relations de compatibilité suivantes entre $\wedge$ et $\diamond$ :
$$\aligned
&\alpha\wedge(\beta\diamond \gamma)=(-1)^{(|\beta|-1)(|\gamma|-1)}\alpha\wedge(\gamma\diamond \beta),\\
&\alpha\diamond(\beta\wedge \gamma)=(\alpha\diamond \beta)\wedge\gamma,\\
&(\alpha\diamond\beta)\wedge \gamma=(-1)^{(|\beta|-1)|\gamma|}(\alpha\wedge \gamma)\diamond\beta.
\endaligned
$$
\end{itemize}
\end{defn}

Si on pose $[\alpha,\beta]=\alpha\diamond\beta-(-1)^{(|\alpha|-1)(|\beta|-1)}\beta\diamond\alpha$ et $\alpha\pt\beta=\alpha\wedge\beta+(-1)^{|\alpha||\beta|}\beta\wedge\alpha$, ces relations impliquent que $(\mathcal G,\pt,[~,~])$ est une alg\`ebre de Gerstenhaber, c'est \`a dire que la relation de Leibniz est vraie :
$$
[\alpha,\beta\pt\gamma]=[\alpha,\beta]\pt\gamma+(-1)^{|\beta|(|\alpha|-1)}\beta\pt[\alpha,\gamma].
$$

Mais on obtient aussi les deux relations suivantes :
$$\aligned
\alpha\wedge[\beta,\gamma]&=0,\\
[\alpha,\beta\wedge\gamma]&=[\alpha,\beta]\wedge\gamma.
\endaligned
$$

Dans \cite{[Ag]}, Aguiar a propos\'e une d\'efinition d'alg\`ebre pr\'e-Gerstenhaber avec d'autres conditions de compatibilit\'es. Pour une alg\`ebre \`a droite, sa d\'efinition est \'equivalente aux conditions de compatibilit\'es suivantes :
$$\aligned
{[\alpha,\beta]}\wedge \gamma&=\alpha\diamond(\beta\wedge \gamma)-(-1)^{(|\alpha|-1)|\beta|}\beta\wedge(\alpha\diamond \gamma),\\
(\alpha\pt\beta)\diamond \gamma &=(-1)^{|\alpha|(|\gamma|-1)}\alpha\wedge(\beta\diamond \gamma)
+(-1)^{(|\alpha|+|\gamma|-1)|\beta|}\beta\wedge(\alpha\diamond\gamma).
\endaligned
$$
Il a aussi construit une telle structure sur $T^+(\mathfrak g)$, si $\mathfrak g$ est une alg\`ebre pr\'e-Lie.

Les deux relations ci-dessus sont une cons\'equence de nos relations de compatibilit\'es, notre notion d'alg\`ebre pr\'e-Gerstenahber est plus stricte que celle d'Aguiar.\\

La premi\`ere \'etape de la construction d'une $preG_\infty$ alg\`ebre consiste, comme ci-dessus \`a construire
la $Z_\infty$ alg\`ebre enveloppante $(\mathcal H,\delta,D)$ de l'alg\`ebre pr\'e-commutative $(\mathcal G,\wedge)$. Cette construction ne d\'epend pas des conditions de compatibilit\'es.\\

La deuxi\`eme \'etape, que nous allons pr\'esenter dans la section suivante, consiste \`a prolonger la loi $\diamond$ en une loi $R_2$ sur $\mathcal H$, de telle fa\c con que $(\mathcal H,R_2,D)$ soit une alg\`ebre pr\'e-Lie diff\'erentielle. Mais les conditions de compatibilit\'es propos\'ees par Aguiar sont insuffisantes pour que $D$ soit une d\'erivation de $R_2$, alors que les conditions propos\'ees dans la d\'efinition ci-dessus garantissent cette propri\'et\'e.\\

Avant de pr\'esenter cette construction, donnons un exemple
d'alg\`ebre pr\'e-Gerstenhaber.

\begin{ex}

Soit $\mathcal{G}$ l'espace des  formes différentielles sur une variété $M$. Si $\alpha$ est une $k$-forme, on d\'efinit le degré de $\alpha$ par : $|\alpha|=k+1$.

Soient $\alpha$ et $\beta$ deux formes différentielles, on définit:
$$
\alpha\curlywedge\beta=\displaystyle\frac{1}{|\beta|}\alpha\wedge d\beta~~\qquad~\hbox{ et }~\qquad~\alpha\diamond\beta=\alpha\wedge\beta.
$$

Alors, on vérifie que $|\curlywedge|=0$ et $|\diamond|=-1$.

Pour $\alpha$, $\beta$ et $\gamma$ dans $\mathcal{G}$, on vérifie aussi que :
$$
(\alpha\curlywedge \beta)\curlywedge \gamma=\alpha\curlywedge(\beta\curlywedge\gamma)+(-1)^{|\beta||\gamma|}\alpha\curlywedge(\gamma\curlywedge\beta),
$$
et
$$
(\alpha\diamond \beta)\diamond \gamma-\alpha\diamond
(\beta\diamond\gamma)=(-1)^{(|\beta|-1)(|\gamma|-1)}\Big((\alpha\diamond\gamma)\diamond \beta -\alpha\diamond(\gamma\diamond\beta)\Big).
$$

Les relations de compatibilités entre $\curlywedge$ et $\diamond$ sont aussi vérifiées :
$$\aligned
\alpha\curlywedge(\beta\diamond \gamma)&=(-1)^{(|\beta|-1)(|\gamma|-1)}\alpha\curlywedge(\gamma\diamond \beta),\\
\alpha\diamond(\beta\curlywedge \gamma)&=(\alpha\diamond \beta)\curlywedge\gamma,\\
(\alpha\diamond\beta)\curlywedge \gamma&=(-1)^{(|\beta|-1)|\gamma|}(\alpha\curlywedge \gamma)\diamond\beta.
\endaligned
$$

Ainsi, $(\mathcal{G},\curlywedge,\diamond)$ est bien une algèbre pré-Gerstenhaber.\\
\end{ex}

%----------------------------------------------

\section{L'alg\`ebre pr\'e-Lie diff\'erentielle $(\mathcal H,R_2,D)$}

\

Soit $(\mathcal G,\wedge,\diamond)$ une alg\`ebre pr\'e-Gerstenhaber gradu\'ee. Puisque $\wedge$ est une loi pr\'e-commutative, on lui associe une cog\`ebre de Leibniz codiff\'erentielle $(\mathcal H,\delta,D)$. Dans cette section, on montre que $\mathcal H$ est aussi munie d'une structure d'alg\`ebre pr\'e-Lie diff\'erentielle.\\

Pour cela, on prolonge $\diamond$ \`a $\mathcal{H}$ en $R_2$ de telle fa\c con que $(\mathcal{H},R_2,D)$ soit une algèbre pré-Lie différentielle. Soient $X=\alpha_1\otimes...\otimes\alpha_p$, et $Y=\alpha_{p+1}\otimes...\otimes\alpha_{p+q}$ deux \'el\'ements de $\mathcal H$, on pose :
$$\aligned
&R_2(X,Y)=\\&=(\alpha_{1}\diamond\alpha_{p+1})\otimes\sum_{\sigma\in Sh_{p-1,q-1}}(-1)^{(\alpha_2+\dots+\alpha_p)\alpha_{p+1}}\varepsilon_\alpha(\sigma^{-1})\alpha_{\sigma^{-1}(2)}\otimes \cdots\widehat{ _{p+1}}\dots\otimes\alpha_{\sigma^{-1}(p+q)}\\
&+\hskip-0.5cm\sum_{\begin{smallmatrix}2\leq k\leq p\\ \sigma\in
Sh_{p-k,q-1}\end{smallmatrix}}\hskip-0.5cm(-1)^{(\alpha_{k+1}+\dots+\alpha_p)\alpha_{p+1}}\varepsilon_\alpha(\sigma^{-1})
\alpha_{1}\otimes\alpha_{2}\otimes\dots\otimes\alpha_{k-1}\otimes[\alpha_{k},\alpha_{p+1}]\otimes
\alpha_{\sigma^{-1}(k+1)}\otimes\dots\\&\hskip3cm\dots\otimes\cdots\widehat{_{p+1}}\dots\otimes\alpha_{\sigma^{-1}(p+q)},
\endaligned
$$
avec : $[\alpha_{k},\alpha_{p+1}]=\alpha_{k}\diamond\alpha_{p+1}-(-1)^{\alpha_{k}\alpha_{p+1}}\alpha_{p+1}\diamond\alpha_{k}$.\\

On rappelle que $deg(X)=deg(\alpha_1\otimes\dots\otimes\alpha_p)=\alpha_1+\cdots+\alpha_p=x$.\\

\begin{thm}

\

Le triplet $(\mathcal{H},R_2,D)$ est une alg\`ebre pré-Lie diff\'erentielle gradu\'ee.\\
\end{thm}

\begin{proof}

\

Montrons d'abord que $(\mathcal{H},R_2)$ est une algèbre pré-Lie graduée. Soient
$$\aligned
X&=\alpha_1\otimes\dots\otimes\alpha_p,\\
Y&=\alpha_{p+1}\otimes\dots\otimes\alpha_{p+q},\\
Z&=\alpha_{p+q+1}\otimes\dots\otimes\alpha_{p+q+r}
\endaligned
$$
trois \'el\'ements de $\mathcal H$. Vérifions la relation $(\ast)$
suivante :
$$
R_2\big(R_2(X,Y),Z\big)-R_2\big(X,R_2(Y,Z)\big)-(-1)^{yz}\Big(R_2\big(R_2(X,Z),Y\big)-R_2\big(X,R_2(Z,Y)\big)\Big)=0.
$$
Dans cette relation, il appara\^\i t 4 types de termes :
\begin{itemize}
\item[1.] Dans $(\ast)$, il appara\^it des termes avec deux $\diamond$ :
$$
(\alpha_1\diamond\alpha_{p+1})\diamond\alpha_{p+q+1},\quad\alpha_1\diamond(\alpha_{p+1}\diamond\alpha_{p+q+1}),\quad(\alpha_1\diamond\alpha_{p+q+1})\diamond\alpha_{p+1},\quad \alpha_1\diamond(\alpha_{p+q+1}\diamond\alpha_{p+1}).
$$
Ces termes apparaissent sous la forme $\pm terme\otimes\alpha_{\sigma^{-1}(2)}\otimes\dots\alpha_{\sigma^{-1}(p+q+r)}$ o\`u $\sigma$ est un shuffle de $\{1,\dots,p+q+r\}\setminus\{1,p+1,p+q+1\}$ : $\sigma\in Sh_{p-1,q-1,r-1}$. Plus pr\'ecis\'ement, on pose :
$$
\varepsilon_\sigma=\varepsilon_\alpha(\sigma)(-1)^{(x-\alpha_1)\alpha_{p+1}}(-1)^{(x+y-\alpha_1-\alpha_{p+1})\alpha_{p+q+1}},
$$
et
$$
A_\sigma=\alpha_{\sigma^{-1}(2)}\otimes\dots\widehat{_{p+1}}\dots\widehat{_{p+q+1}}\dots\otimes\alpha_{\sigma^{-1}(p+q+r)}.
$$
La contribution des termes correspondants est $C\otimes(\varepsilon_\sigma A_\sigma)$ avec :
$$\aligned
C&=(\alpha_1\diamond\alpha_{p+1})\diamond\alpha_{p+q+1}-\alpha_1\diamond(\alpha_{p+1}\diamond\alpha_{p+q+1})-\\
&\hskip 1cm-\varepsilon'\big((\alpha_1\diamond\alpha_{p+q+1})\diamond\alpha_{p+1}- \alpha_1\diamond(\alpha_{p+q+1}\diamond\alpha_{p+1})\big),
\endaligned
$$
o\`u $\varepsilon'$ est le signe :
$$
\varepsilon'=(-1)^{yz}(-1)^{(y-\alpha_{p+1})(z-\alpha_{p+q+1})}(-1)^{\alpha_{p+q+1}(y-\alpha_{p+1})+\alpha_{p+1}(z-\alpha_{p+q+1})}=(-1)^{\alpha_{p+1}\alpha_{p+q+1}}.
$$
Ces termes se simplifient gr\^ace \`a la relation pr\'e-Lie de $\diamond$.\\
\item[2.] Dans $(\ast)$, il appara\^it des termes avec un double crochet ou un $\diamond$ dans un crochet :
$$
[[\alpha_k,\alpha_{p+1}],\alpha_{p+q+1}],\quad[\alpha_k,\alpha_{p+1}\diamond\alpha_{p+q+1}],\quad[[\alpha_k,\alpha_{p+q+1}],\alpha_{p+1}],\quad
[\alpha_k,\alpha_{p+q+1}\diamond\alpha_{p+1}].
$$
Ces termes apparaissent pour $1<k\leq p$ sous la forme
$$\aligned
&\pm (\alpha_1\otimes\dots\otimes\alpha_{k-1})\otimes terme\otimes\alpha_{\sigma^{-1}(k+1)}\otimes\dots\widehat{_{p+1}}\dots\widehat{_{p+q+1}}\dots\otimes\alpha_{\sigma^{-1}(p+q+r)}\\
=&\pm A_k\otimes terme\otimes B_\sigma.
\endaligned
$$
o\`u $\sigma$ est dans $Sh_{p-k,q-1,r-1}$ agissant sur $\{k+1,\dots,p+q+r\}\setminus\{p+1,p+q+1\}$. On obtient donc les termes $\varepsilon_\alpha(\sigma) A_k\otimes C\otimes B_\sigma$ avec, le m\^eme $\varepsilon'$ que ci-dessus :
$$\aligned
C&=
[[\alpha_k,\alpha_{p+1}],\alpha_{p+q+1}]-[\alpha_k,\alpha_{p+1}\diamond\alpha_{p+q+1}]-\varepsilon'[[\alpha_k,\alpha_{p+q+1}],\alpha_{p+1}]+\\
&\hskip 9cm+\varepsilon'
[\alpha_k,\alpha_{p+q+1}\diamond\alpha_{p+1}]\\
&=[[\alpha_k,\alpha_{p+1}],\alpha_{p+q+1}]-[\alpha_k,[\alpha_{p+1},\alpha_{p+q+1}]]-(-1)^{\alpha_{p+1}\alpha_{p+q+1}}[[\alpha_k,\alpha_{p+q+1}],\alpha_{p+1}]\\
&=0.
\endaligned
$$

\item[3.] Dans $(\ast)$, il appara\^it des termes de la forme
$$
\dots\otimes[\alpha_k,\alpha_{p+1}]\otimes\dots\otimes[\alpha_\ell,\alpha_{p+q+1}]\otimes\cdots,\quad\dots\otimes[\alpha_\ell,\alpha_{p+q+1}]\otimes\dots\otimes[\alpha_k,\alpha_{p+1}]\otimes\cdots.
$$
Plus pr\'ecis\'ement :

- Dans $R_2\big(R_2(X,Y),Z\big)$, pour tout $k\in\{2,\dots,p\}$, les termes qui apparaissent sont :

$(1.1): \dots\otimes[\alpha_k,\alpha_{p+1}]\otimes\dots\otimes[\alpha_\ell,\alpha_{p+q+1}]\otimes\cdots$ , avec $k<\ell\leq p$,

$(1.2): \dots\otimes[\alpha_k,\alpha_{p+1}]\otimes\dots\otimes[\alpha_\ell,\alpha_{p+q+1}]\otimes\cdots$ , avec $p+1<\ell\leq p+q$,

$(1.3): \dots\otimes[\alpha_\ell,\alpha_{p+q+1}]\otimes\dots\otimes[\alpha_k,\alpha_{p+1}]\otimes\cdots$ , avec $1<\ell< k$.\\

- Dans $R_2\big(X,R_2(Y,Z)\big)$, pour tout $k\in\{2,\dots,p\}$, les termes qui apparaissent sont :

$(2.1): \dots\otimes[\alpha_k,\alpha_{p+1}]\otimes\dots\otimes[\alpha_\ell,\alpha_{p+q+1}]\otimes\cdots$ , avec $p+1<\ell\leq p+q$.\\

- Dans $R_2\big(R_2(X,Z),Y\big)$, pour tout $k\in\{2,\dots,p\}$, les termes qui apparaissent sont :

$(3.1): \dots\otimes[\alpha_k,\alpha_{p+q+1}]\otimes\dots\otimes[\alpha_\ell,\alpha_{p+1}]\otimes\cdots$ , avec $k<\ell\leq p$,

$(3.2): \dots\otimes[\alpha_k,\alpha_{p+q+1}]\otimes\dots\otimes[\alpha_\ell,\alpha_{p+1}]\otimes\cdots$ , avec $p+q+1<\ell\leq p+q+r$,

$(3.3): \dots\otimes[\alpha_\ell,\alpha_{p+1}]\otimes\dots\otimes[\alpha_k,\alpha_{p+q+1}]\otimes\cdots$ , avec $1<\ell< k$.\\

- Dans $R_2\big(X,R_2(Z,Y)\big)$, pour tout $k\in\{2,\dots,p\}$, les termes qui apparaissent sont :

$(4.1): \dots\otimes[\alpha_k,\alpha_{p+q+1}]\otimes\dots\otimes[\alpha_\ell,\alpha_{p+1}]\otimes\cdots$ , avec $p+q+1<\ell\leq p+q+r$.\\

Il est clair que $(1.2)-(2.1)=0$ et $(3.2)-(4.1)=0$. En utilisant la commutativité des shuffles, on vérifie que
$(1.1)=(3.3)$ et $(1.3)=(3.1)$.\\

\item[4.] Enfin, dans $(\ast)$, il appara\^it des termes de la forme
$$
\alpha_1\diamond\alpha_{p+1}\otimes\dots\otimes[\alpha_k,\alpha_{p+q+1}]\otimes\cdots,\quad \alpha_1\diamond\alpha_{p+q+1}\otimes\dots
\otimes[\alpha_k,\alpha_{p+1}]\otimes\dots.
$$
Plus pr\'ecis\'ement,\\

- Dans $R_2\big(R_2(X,Y),Z\big)$, les termes qui apparaissent sont :

$(1.1)': \alpha_1\diamond\alpha_{p+1}\otimes\dots\otimes[\alpha_k,\alpha_{p+q+1}]\otimes\cdots$ , avec $1<k\leq p$,

$(1.2)': \alpha_1\diamond\alpha_{p+1}\otimes\dots\otimes[\alpha_k,\alpha_{p+q+1}]\otimes\cdots$ , avec $p+1<k\leq p+q$,

$(1.3)': \alpha_1\diamond\alpha_{p+q+1}\otimes\dots\otimes[\alpha_k,\alpha_{p+1}]\otimes\cdots$ , avec $1<k\leq p$.\\

\n- Dans $R_2\big(X,R_2(Y,Z)\big)$, les termes qui apparaissent sont:

$(2.1)': \alpha_1\diamond\alpha_{p+1}\otimes\dots\otimes[\alpha_k,\alpha_{p+q+1}]\otimes\cdots$ , avec $p+1<k\leq p+q$.\\

- Dans $R_2\big(R_2(X,Z),Y\big)$, les termes qui apparaissent sont :

$(3.1)': \alpha_1\diamond\alpha_{p+q+1}\otimes\dots\otimes[\alpha_k,\alpha_{p+1}]\otimes\cdots$ , avec $1<k\leq p$,

$(3.2)': \alpha_1\diamond\alpha_{p+q+1}\otimes\dots\otimes[\alpha_k,\alpha_{p+1}]\otimes\cdots$ , avec $p+q+1<k\leq p+q+r$,

$(3.3)': \alpha_1\diamond\alpha_{p+1}\otimes\dots\otimes[\alpha_k,\alpha_{p+q+1}]\otimes\cdots$ , avec $1< k\leq p$.\\

- Dans $R_2\big(X,R_2(Z,Y)\big)$, les termes qui apparaissent sont :

$(4.1)': \alpha_1\diamond\alpha_{p+q+1}\otimes\dots\otimes[\alpha_k,\alpha_{p+1}]\otimes\cdots$ , avec $p+q+1<k\leq p+q+r$.\\

Il est clair que $(1.2)'-(2.1)'=0$ et $(3.2)'-(4.1)'=0$. En utilisant la commutativité des battements, on vérifie que $(1.1)'=(3.3)'$ et $(1.3)'=(3.1)'$.\\
\end{itemize}

Montrons maintenant que la différentielle $D$ est une dérivation de $R_2$. Pour $X=\alpha_1\otimes\dots\otimes\alpha_p$ et $
Y=\alpha_{p+1}\otimes\dots\otimes\alpha_{p+q}$, on vérifie que
$$
D\circ R_2(X,Y)=R_2\big(D(X),Y\big)+(-1)^{x}R_2\big(X,D(Y)\big)
$$

Rappelons que
$$\aligned
D(\alpha_1\otimes...\otimes&\alpha_p)=
D_2(\alpha_1,\alpha_2)\otimes \alpha_3\otimes\dots \otimes\alpha_n
\\& + \displaystyle\sum_{k=2}^{n-1}(-1)^{
\sum_{i<k}\alpha_i}\alpha_1\otimes\dots\otimes \alpha_{k-1}\otimes
D_2\circ \mu_2(\alpha_k,\alpha_{k+1})\otimes
\alpha_{k+2}\otimes\dots\otimes \alpha_n.\endaligned
$$

On pose $m_2=D_2\circ\mu_2$ et on note symboliquement $D=D_2+m_2$.\\

De m\^eme, $R_2$ s'\'ecrit:
$$\aligned
&R_2(X,Y)=\\&=(\alpha_{1}\diamond\alpha_{p+1})\otimes\sum_{\sigma\in Sh_{p-1,q-1}}(-1)^{(\alpha_2+\dots+\alpha_p)\alpha_{p+1}}\varepsilon_\alpha(\sigma^{-1})\alpha_{\sigma^{-1}(2)}\otimes \cdots\widehat{_{p+1}}\dots\otimes\alpha_{\sigma^{-1}(p+q)}\\
&+\hskip-0.5cm\sum_{\begin{smallmatrix}2\leq k\leq p\\ \sigma\in
Sh_{p-k,q-1}\end{smallmatrix}}\hskip-0.5cm(-1)^{(\alpha_{k+1}+\dots+\alpha_p)\alpha_{p+1}}\varepsilon_\alpha(\sigma^{-1})
\alpha_{1}\otimes\alpha_{2}\otimes\dots\otimes\alpha_{k-1}\otimes[\alpha_{k},\alpha_{p+1}]\otimes
\alpha_{\sigma^{-1}(k+1)}\otimes\\
&\hskip3cm\dots\otimes\cdots\widehat{_{p+1}}\dots\otimes\alpha_{\sigma^{-1}(p+q)}.
\endaligned
$$

Ce qu'on note symboliquement $R_2=\diamond+[~,~]$.\\

La relation \`a montrer s'\'ecrit donc :
$$
(1)=D\circ R_2=R_2\circ(D\otimes id)+R_2\circ(id\otimes D)=(2)+(3).
$$
Avec nos notations symboliques, on trouve :
$$\aligned
(1)&=D_2\circ\diamond+D_2\circ[~,~]+m_2\circ[~,~]+\diamond\otimes m_2+D_2\otimes[~,~]+\big(m_2\otimes[~,~]+[~,~]\otimes m_2\big)\\
&=(1.1)+(1.2)+(1.3)+(1.4)+(1.5)+(1.6).
\endaligned
$$
De m\^eme:
$$\aligned
(2)&=\diamond\circ(D_2\otimes id)+[~,~]\circ(m_2\otimes id)+\diamond\otimes m_2+D_2\otimes[~,~]+\big(m_2\otimes[~,~]+[~,~]\otimes m_2\big)\\
&=(2.1)+(2.2)+(2.3)+(2.4)+(2.5),\\
%\hbox{et }&\\
(3)&=\diamond\circ(id\otimes D_2)+[~,~]\circ(id\otimes D_2)+\diamond\otimes m_2+[~,~]\otimes m_2\\
&=(3.1)+(3.2)+(3.3)+(3.4).
\endaligned
$$

La preuve consiste \`a v\'erifier les 6 \'egalit\'es suivantes:
$$\aligned
(1.1)&=(2.1)+(3.1),\\
(1.2)&=0,\\
(1.3)&=(2.2)+(3.2),\\
(1.4)&=(2.3)+(3.3),\\
(1.5)&=(2.4),\\
(1.6)&=(2.5)+(3.4).
\endaligned
$$

Les trois premi\`eres font intervenir les conditions de compatibilit\'es et nous les d\'etaillons ci-dessous.\\

\begin{itemize}
\item[1.] $(1.1)=(2.1)+(3.1)$.

Les termes de (1.1) sont :
$$\aligned
(1.1)&=\sum_{\sigma\in Sh_{p-1,q-1}}\varepsilon_\alpha(\sigma)(-1)^{\alpha_{p+1}(\alpha_2+\dots+\alpha_p)}D_2(\alpha_{1}\diamond\alpha_{p+1},\alpha_{\sigma^{-1}(2)})\otimes\cdots\widehat{_{p+1}}\dots\otimes\alpha_{\sigma^{-1}(p+q)}.\\
\endaligned
$$
Comme $\sigma$ est un shuffle, on a $\sigma^{-1}(2)=2$ ou $\sigma^{-1}(2)=p+2$.

\noindent De m\^eme,
$$
(2.1)=(D_2(\alpha_{1},\alpha_2)\diamond\alpha_{p+1})\otimes\sum_{\sigma\in
Sh_{p-2,q-1}}\hskip-0.4cm
(-1)^{(\alpha_3+\dots+\alpha_p)\alpha_{p+1}}
\varepsilon_\alpha(\sigma)
\alpha_{\sigma^{-1}(3)}\otimes\cdots\widehat{_{p+1}}\dots\otimes\alpha_{\sigma^{-1}(p+q)},
$$
et
$$\aligned
(3.1)&=(\alpha_1\diamond D_2(\alpha_{p+1},\alpha_{p+2}))\otimes\sum_{\sigma\in Sh_{p-2,q-1}} (-1)^{(\alpha_2+\dots+\alpha_p)(\alpha_{p+1}+\alpha_{p+2})} \varepsilon_\alpha(\sigma)\\
&\hskip 5cm\alpha_{\sigma^{-1}(3)}\otimes\cdots\widehat{_{p+1~p+2}}\dots\otimes\alpha_{\sigma^{-1}(p+q)}.
\endaligned
$$
\noindent Avec les relations de compatibilit\'e :
$$
\alpha\diamond(\beta\wedge \gamma)=(\alpha\diamond
\beta)\wedge\gamma
\hskip0.7cm\hbox{et}\hskip0.6cm(\alpha\diamond\beta)\wedge\gamma=(-1)^{(|\beta|-1)|\gamma|}
(\alpha\wedge\gamma)\diamond\beta,
$$
ceci prouve la relation $(1.1)=(2.1)+(3.1)$.\\

\item[2.] $(1.2)=0$.

Les termes de (1.2) sont :
$$
(1.2)=D_2(\alpha_1,[\alpha_2,\alpha_{p+1}])\otimes\hskip-0.5cm
\sum_{\sigma\in Sh_{p-2,q-1}}\hskip-0.5cm
\varepsilon_\alpha(\sigma^{-1})(-1)^{\alpha_{p+1}(\alpha_3+\dots+\alpha_p)}\alpha_{\sigma^{-1}(3)}\otimes\cdots
\widehat{_{p+1}}\dots\otimes\alpha_{\sigma^{-1}(p+q)}.
$$
Avec la relation de compatibilit\'e :
$$
\alpha\wedge(\beta\diamond
\gamma)=(-1)^{(|\beta|-1)(|\gamma|-1)}\alpha\wedge(\gamma\diamond
\beta),
$$
ceci prouve la relation $(1.2)=0$.\\

\item[3.] $(1.3)=(2.2)+(3.2)$.

Les termes de (1.3) sont :
$$\aligned
(1.3)&=\sum_{3\leq k\leq p}(-1)^{\alpha_{p+1}(\alpha_{k+1}+\dots+\alpha_p)}\alpha_1\otimes\dots\otimes \alpha_{k-2}\otimes m_2(\alpha_{k-1},[\alpha_{k},\alpha_{p+1}])\otimes\\
&\hskip 1cm\otimes\sum_{\sigma\in Sh_{p-k,q-1}}\varepsilon_\alpha(\sigma)\alpha_{\sigma^{-1}(k+1)}\otimes\cdots
\widehat{_{p+1}}\dots\otimes\alpha_{\sigma^{-1}(p+q)}+\\
&+(-1)^{\alpha_{p+1}(\alpha_{k}+\dots+\alpha_p)}\alpha_1\otimes\dots\otimes \alpha_{k-2}\otimes\sum_{\sigma\in Sh_{p-k-1,q-1}}\varepsilon_\alpha(\sigma) m_2([\alpha_{k-1},\alpha_{p+1}],\alpha_{\sigma^{-1}(k)})\otimes\\
&\hskip 1cm\otimes\alpha_{\sigma^{-1}(k+1)}\otimes\cdots\widehat{_{p+1}}\dots\otimes\alpha_{\sigma^{-1}(p+q)}.
\endaligned
$$

Comme $\sigma$ est un shuffle, on a $\sigma^{-1}(k)=k$ ou $\sigma^{-1}(k)=p+2$.

De m\^eme,
$$\aligned
(2.2)&=\sum_{3\leq k\leq p} (-1)^{\alpha_{p+1}(\alpha_{k+1}+\dots+\alpha_p)}\alpha_{1}\otimes\dots\otimes\alpha_{k-2}\otimes[m_2(\alpha_{k-1},\alpha_{k}),\alpha_{p+1}]\otimes\\
&\hskip 2cm\otimes \sum_{\sigma\in Sh_{p-k,q-1}}\varepsilon_\alpha(\sigma)\alpha_{\sigma^{-1}(k+1)}\otimes\dots
\widehat{_{p+1}}\dots\otimes\alpha_{\sigma^{-1}(p+q)}.\\
\endaligned
$$
Et
$$\aligned
(3.2)&=\sum_{3\leq k\leq p}(-1)^{(\alpha_{p+1}+\alpha_{p+2})(\alpha_{k+1}+\dots+\alpha_p)}\alpha_{1}\otimes\dots\otimes\alpha_{k-2}\otimes[\alpha_{k-1},D_2(\alpha_{p+1},\alpha_{p+2})]\otimes\\
&\hskip 2cm\otimes \sum_{\sigma\in Sh_{p-k,q-1}}\varepsilon_\alpha(\sigma) \alpha_{\sigma^{-1}(k+1)}\otimes\cdots\widehat{_{p+1~p+2}}\dots\otimes\alpha_{\sigma^{-1}(p+q)}.
\endaligned
$$
Avec la relation de compatibilit\'e :
$$
[\alpha,\beta\wedge\gamma]=[\alpha,\beta]\wedge\gamma,
$$
ceci prouve la relation $(1.3)=(2.2)+(3.2)$.\\
\end{itemize}

Il reste 3 relations ne d\'ependant pas des relations de compatibilit\'es. On expose ici de fa\c con moins d\'etaill\'ee leur preuve :\\
\begin{itemize}
\item[4.] $(1.4)=(2.3)+(3.3)$.

Les termes de (1.4) sont :
$$\aligned
(1.4)=&(\alpha_{1}\diamond\alpha_{p+1})\otimes\hskip-0.5cm\sum_{\begin{smallmatrix}k\in\{2,\dots, p+q-1\}\setminus\{p+1\}\\ \sigma\in Sh_{p-1,q-1}\end{smallmatrix}}\hskip-0.5cm (-1)^{\alpha_1+\alpha_{p+1}+\alpha_{\sigma^{-1}(2)}+\dots+\alpha_{\sigma^{-1}(k-1)}}(-1)^{\alpha_{p+1}(\alpha_2+\dots+\alpha_p)}\\
&\hskip 1cm\varepsilon_\alpha(\sigma)\alpha_{\sigma^{-1}(2)}\otimes\dots\otimes m_2(\alpha_{\sigma^{-1}(k)},\alpha_{\sigma^{-1}(k+1)})\otimes\cdots\widehat
{_{p+1}}\dots\otimes\alpha_{\sigma^{-1}(p+q)}.
\endaligned
$$
Dans $(1.4)$, les termes qui correspondent à ($\sigma^{-1}(k)\in\{2,\dots,p\}$ et $\sigma^{-1}(k+1)\in\{p+2,\dots,p+q\}$) et ($\sigma^{-1}(k+1)\in\{2,\dots,p\}$ et $\sigma^{-1}(k)\in\{p+2,\dots,p+q\}$) se simplifient entre eux. Il ne reste que les termes qui correspondent à
\subitem $\sigma^{-1}(k),\sigma^{-1}(k+1)\in\{2,\dots,p\}$ : ces termes correspondent \`a ceux de (2.3),
\subitem $\sigma^{-1}(k),\sigma^{-1}(k+1)\in\{p+2,\dots,p+q\}$ : ces termes sont ceux de (3.3).\\

\item[5.] $(1.5)=(2.4)$.

Les termes de (1.5) sont :
$$\aligned
(1.5)&=D_2(\alpha_1,\alpha_2)\otimes\displaystyle\sum_{3\leq
k\leq p}(-1)^{\alpha_{p+1}(\alpha_{k+1}+\dots+\alpha_p)}
\alpha_{3}\otimes\dots\otimes\alpha_{k-1}\otimes[\alpha_{k},\alpha_{p+1}]\otimes
\\&\otimes\sum_{\sigma\in Sh_{p-k,q-1}}\varepsilon_\alpha(\sigma)\alpha_{\sigma^{-1}(k+1)}\otimes\dots
\widehat{_{p+1}}\dots\otimes\alpha_{\sigma^{-1}(p+q)}.\endaligned
$$
Ceci est exactement le terme $(2.4)$.\\

\item[6.] $(1.6)=(2.5)+(3.4)$.

Les termes de (1.6) sont :
$$\aligned
(1.6)=&\sum_{2\leq k\leq p}(-1)^{\alpha_{p+1}(\alpha_{k+1}+\dots+\alpha_p)}
\alpha_{1}\otimes\dots\otimes\alpha_{k-1}\otimes[\alpha_{k},\alpha_{p+1}]\otimes\sum_{\begin{smallmatrix}\sigma\in Sh_{p-k,q-1}
\\ k+1\leq j\leq p+q-1\end{smallmatrix}}\varepsilon_\alpha(\sigma)\\
&\hskip 1cm\alpha_{\sigma^{-1}(k+1)}
\otimes\dots\otimes m_2(\alpha_{\sigma^{-1}(j)},\alpha_{\sigma^{-1}(j+1)})\otimes\cdots\widehat
{_{p+1}}\dots\otimes\alpha_{\sigma^{-1}(p+q)}.
\endaligned
$$
Dans $(1.6)$, les termes qui correspondent à ($\sigma^{-1}(j)\in\{2,\dots,p\}$ et $\sigma^{-1}(j+1)\in\{p+2,\dots,p+q\}$) et ($\sigma^{-1}(j+1)\in\{2,\dots,p\}$ et $\sigma^{-1}(j)\in\{p+2,\dots,p+q\}$) se simplifient entre eux. Il ne reste que les termes qui correspondent à
\subitem $\sigma^{-1}(j),\sigma^{-1}(j+1)\in\{2,\dots,p\}$ : ces termes correspondent \`a ceux de (2.5),
\subitem $\sigma^{-1}(j),\sigma^{-1}(j+1)\in\{p+2,\dots,p+q\}$ : ces termes sont ceux de (3.4).\\
\end{itemize}

\end{proof}

Maintenant, puisque $(\mathcal H,R_2,D)$ est une alg\`ebre pr\'e-Lie diff\'erentielle gradu\'ee, on construit comme dans la section pr\'ec\'edente, sa $preL_\infty$ algèbre enveloppante $(\mathcal{H}[1]\otimes S(\mathcal{H}[1]),\Delta,Q)$. Explicitement, dans $\mathcal H[1]$, le degré est $deg'(X)=deg(X)-1=x'$, on pose
$$
R'_2(X,Y)=(-1)^{x'}R_2(X,Y)~~~\text{ et }~~\ell_2(X,Y)=R_2'(X,Y)+(-1)^{x'y'}R_2'(Y,X).~
$$

On prolonge ensuite \`a $\mathcal{H}[1]\otimes S(\mathcal{H}[1])$, $D$ et $R'_2$ en $m$ et $R$ par :
$$\aligned
m(X_0\otimes X_1\dots X_n)&=D(X_0)\otimes X_1\dots X_n+\\
&\hskip 1cm+(-1)^{x_0'}\sum_{j=1}^{n}\varepsilon_{x'}\Big(\begin{smallmatrix}
x_1\dots x_n\\ x_j~x_1\dots\hat{_j}\dots x_n\end{smallmatrix}\Big)X_0\otimes
D(X_j) X_1\dots\widehat{_j}\dots X_n\\
&=D(X_0)\otimes X_1\dots X_n+(-1)^{x_0'}X_0\otimes m'(X_1\dots X_n),
\endaligned
$$
où $m'$ est la codérivation donnée comme ci-dessus dans le cas des
algèbres de Gerstenhaber.

De m\^eme,
$$\aligned
&R(X_0\otimes X_1\dots
X_n)=\sum_{i=1}^n\varepsilon_{x'}\Big(\begin{smallmatrix} x_1\dots
x_n\\ x_i~x_1\dots\hat{_i}\dots x_n\end{smallmatrix}\Big)R_2'(X_0,
X_i) \otimes
X_{1}\dots \widehat{_i}\dots X_{n}+\\
&+(-1)^{x_0'}\sum_{i<j}\varepsilon_{x'}\Big(\begin{smallmatrix}
x_1\dots x_n\\ x_i~x_j~x_1\dots\hat{_i}\dots\hat{_j}\dots
x_n\end{smallmatrix}\Big) X_0\otimes
\ell_2(X_i, X_j)X_{1}\dots\widehat{_{ij}}\dots X_{n}\\
&=\sum_{i=1}^n\varepsilon_{x'}\Big(\begin{smallmatrix}
x_1\dots x_n\\ x_i~x_1\dots\hat{_i}\dots x_n\end{smallmatrix}\Big)R_2'(X_0, X_i) \otimes
X_{1}\dots \widehat{_i}\dots X_{n}+(-1)^{x_0'}X_0\otimes\ell'(X_1\dots X_n),
\endaligned
$$
où $\ell'$ est la codérivation donnée comme ci-dessus dans le cas
des algèbres de Gerstenhaber.

Si on pose $Q=(m+R)$, alors $Q$ est une cod\'erivation de $\Delta$ vérifiant $deg'(Q)=1$ et $Q^2=0$. Ainsi :

\begin{cor}

\

Le triplet $\left(\mathcal{H}[1]\otimes S(\mathcal{H}[1]),\Delta,Q=m+R\right)$ est une cogèbre permutative codifférentielle, c'est à dire une $preL_\infty$ alg\`ebre.\\
\end{cor}

Il reste \`a construire sur cette $preL_\infty$ alg\`ebre le coproduit $\kappa$ qui fera de $\mathcal H[1]\otimes S(\mathcal H[1])$ une cog\`ebre de Leibniz. C'est le but de la section suivante.\\

%----------------------

\section{Le coproduit $\kappa$}

\

D'abord, on a vu que $(\mathcal{H},\delta)$ est une cog\`ebre de Leibniz. On rappelle que :
$$
\delta(X)=\delta(x_1\otimes \dots \otimes x_n)=\sum_{1\leq k\leq n-1}(x_1\otimes \dots \otimes x_{k})\bigotimes\mu_{n-k}(x_{k+1}\otimes \dots \otimes x_{n}).
$$
On note simplement le coproduit $\delta$ :
$$
\delta(X)=\sum_{U\otimes V=X}U\otimes \mu V.
$$

On se place maintenant dans $\mathcal{H}[1]$, le coproduit $\delta$ sym\'etris\'e sera not\'e $\kappa$.\\

\begin{defn}

\

Sur $\mathcal H[1]$, on d\'efinit le cocrochet suivant :
$$
\kappa(X_0)=\sum_{U_0\otimes V_0=X_0}(-1)^{u_0'}\Big(U_0\bigotimes\mu V_0+(-1)^{u_0'v_0'}\mu V_0\bigotimes U_0\Big)
$$
Ce cocrochet $\kappa$ se prolonge en un coproduit, toujours not\'e $\kappa$, d\'efini sur $\mathcal{H}[1]\otimes S(\mathcal{H}[1])$ par :
$$\aligned
\kappa(X_0\otimes X_1\dots X_n)&=\sum_{\begin{smallmatrix}
U_0\otimes V_0= X_0\\ I\cup J=\{1,\dots,n\}\end{smallmatrix}}(-1)^{u_0'}\times\Big(\varepsilon_{x'}\Big(\begin{smallmatrix}
u_0v_0x_1\dots x_n\\ u_0~x_I~v_0~x_J\end{smallmatrix}\Big)U_0\otimes X_I\bigotimes\mu V_0.X_J\\
&+\varepsilon_{x'}\Big(\begin{smallmatrix} u_0v_0x_1\dots x_n\\ v_0~x_J~u_0~x_I\end{smallmatrix}\Big)\mu V_0\otimes X_J\bigotimes U_0.
X_I\Big)+(-1)^{x_0'}X_0\otimes \kappa'(X_1... X_n),
\endaligned
$$
o\`u $\kappa'$ est le cocrochet sur $S^+(\mathcal{H}[1])$, d\'efini comme dans la section sur les algèbres de Gerstenhaber par :
\begin{align*}
&\kappa'(X_1\dots X_n)=\sum_{\begin{smallmatrix}1\leq s\leq n\\ I\cup J=\{1,\dots,n\}\setminus\{s\}\end{smallmatrix}}(-1)^{\sum_{i<s}x_i'}\sum_{\begin{smallmatrix}U_s\otimes V_s=X_s\\ U_s,V_s\neq\emptyset\end{smallmatrix} }(-1)^{u_s'}\times\cr&\times\left(\varepsilon_{x'}\left(\begin{smallmatrix} x_1\dots x_n\\ x_I~u_s~v_s~x_J\end{smallmatrix}\right)X_I. U_s\bigotimes \mu V_s . X_J+\varepsilon_{x'}\left(\begin{smallmatrix} x_1\dots x_n\\ x_I~v_s~u_s~x_J\end{smallmatrix}\right)X_I.\mu V_s\bigotimes U_s . X_J\right).
\end{align*}
\end{defn}

Maintenant $\kappa$ est un coproduit de degr\'e 1 qui, en un certain sens, prolonge le coproduit de Leibniz $\delta$, de degr\'e 0, d\'efini sur $\mathcal H$. En fait, on peut dire que $(\mathcal H[1]\otimes S(\mathcal H[1]),\kappa)$ est une cog\`ebre de Leibniz, en tenant compte de ce d\'ecalage de degr\'e. C'est \`a dire :

\begin{prop}

\

Le coproduit $\kappa$ vérifie:
$$
-(id\otimes \kappa)\circ\kappa=\big(\kappa\otimes id+ \tau_{23}\circ(\kappa\otimes id)\big)\circ\kappa.
$$
\end{prop}

\begin{proof}

\

D'une part, on a

$$\aligned
&-(id\otimes\kappa)\circ\kappa(X_0\otimes X_1... X_n)=\\
&=-(id\otimes\kappa)\Big(\sum_{\begin{smallmatrix}U_0\otimes V_0=X_0\\ I\cup
J=\{1,\dots,n\}\end{smallmatrix}}(-1)^{u_0'}\times\Big(\varepsilon_{x'}\left(\begin{smallmatrix}u_0v_0x_1\dots x_n\\ u_0~x_I~v_0~x_J\end{smallmatrix}\right)U_0\otimes X_I\bigotimes\mu V_0.X_J+\\
&\hskip 1cm+\varepsilon_{x'}\left(\begin{smallmatrix} u_0v_0x_1\dots x_n\\ v_0~x_J~u_0~x_I\end{smallmatrix}\right)\mu V_0\otimes X_J\bigotimes U_0.
X_I\Big)+(-1)^{x_0'}X_0\otimes \kappa'(X_1\dots X_n)\Big)\\
&=-\sum_{\begin{smallmatrix}U_0\otimes V_0=X_0\\ I\cup J=\{1,\dots,n\}\end{smallmatrix}}(-1)^{u_0'}\times\Big(\varepsilon_{x'}\left(\begin{smallmatrix}u_0v_0x_1\dots x_n\\ u_0~x_I~v_0~x_J\end{smallmatrix}\right)U_0\otimes X_I\bigotimes\kappa'(\mu V_0.X_J)+\\
&\hskip 1cm+\varepsilon_{x'}\Big(\begin{smallmatrix} u_0v_0x_1\dots x_n\\ v_0~x_J~u_0~x_I\end{smallmatrix}\Big)\mu V_0\otimes X_J\bigotimes
\kappa'(U_0. X_I)\Big)%\\&
-(-1)^{x_0'}X_0\otimes (id\otimes\kappa')\circ\kappa'(X_1\dots X_n)\\
&=(1.1)+(1.2),
\endaligned
$$
o\`u $(1,2)$ est le dernier terme, en $X_0\otimes\kappa'\circ(id\otimes\kappa')$.

D'autre part, on a
$$\aligned
&(\kappa\otimes id)\circ\kappa(X_0\otimes X_1\dots X_n)=\\
&=(\kappa\otimes id)\Big(\sum_{\begin{smallmatrix}U_0\otimes V_0=X_0\\ I\cup J=\{1,\dots,n\}\end{smallmatrix}}(-1)^{u_0'}\times\Big(\varepsilon_{x'}\Big(\begin{smallmatrix}
u_0v_0x_1\dots x_n\\ u_0~x_I~v_0~x_J\end{smallmatrix}\Big)U_0\otimes X_I\bigotimes\mu V_0.X_J\\
&\hskip 1cm+\varepsilon_{x'}\Big(\begin{smallmatrix} u_0v_0x_1\dots x_n\\ v_0~x_J~u_0~x_I\end{smallmatrix}\Big)\mu V_0\otimes X_J\bigotimes U_0.
X_I\Big)+(-1)^{x_0'}X_0\otimes \kappa'(X_1\dots X_n)\Big)\\
&=\sum_{\begin{smallmatrix}U_0\otimes V_0=X_0\\ I\cup J=\{1,\dots,n\}\end{smallmatrix}}(-1)^{u_0'}\times\Big(\varepsilon_{x'}\Big(\begin{smallmatrix}
u_0v_0x_1\dots x_n\\ u_0~x_I~v_0~x_J\end{smallmatrix}\Big)
\kappa(U_0\otimes X_I)\bigotimes\mu V_0.X_J+\\
&\hskip 1cm+\varepsilon_{x'}\Big(\begin{smallmatrix} u_0v_0x_1\dots x_n\\
v_0~x_J~u_0~x_I\end{smallmatrix}\Big)\kappa(\mu V_0\otimes X_J)\bigotimes U_0. X_I\Big)+(-1)^{x_0'}\kappa(X_0)\otimes \kappa'(X_1\dots X_n)+\\
&\hskip 1cm+(-1)^{x_0'}X_0\otimes (\kappa'\otimes id)\circ\kappa'(X_1\dots X_n)\\
&=(2.1)+(2.2),
\endaligned
$$
o\`u $(2,2)$ est le dernier terme, en $X_0\otimes (\kappa'\otimes id)\circ\kappa'$.

Et on a aussi :
$$\aligned
&\tau_{23}\circ(\kappa\otimes id)\circ\kappa(X_0\otimes X_1\dots X_n)=\\
&=\tau_{23}\Big(\sum_{\begin{smallmatrix}U_0\otimes V_0=X_0\\ I\cup J=\{1,\dots,n\}\end{smallmatrix}}(-1)^{u_0'}\times\Big(\varepsilon_{x'}\left(\begin{smallmatrix}
u_0v_0x_1\dots x_n\\ u_0~x_I~v_0~x_J\end{smallmatrix}\right)\kappa(U_0\otimes X_I)\bigotimes\mu V_0.X_J+\\
&\hskip 1cm+\varepsilon_{x'}\left(\begin{smallmatrix} u_0v_0x_1\dots x_n\\ v_0~x_J~u_0~x_I\end{smallmatrix}\right)\kappa(\mu V_0\otimes X_J)\bigotimes U_0. X_I\Big)+(-1)^{x_0'}\kappa(X_0)\otimes\kappa'(X_1\dots X_n)\Big)+\\
&\hskip 1cm+(-1)^{x_0'}X_0\otimes \tau_{23}\circ(\kappa'\otimes id)\circ\kappa'(X_1\dots X_n)\\
&=(3.1)+(3.2),
\endaligned
$$
o\`u $(3.2)$ est le terme en $X_0\otimes \tau_{23}\circ(\kappa'\otimes id)\circ\kappa'$.

L'identité de coJacobi est vérifiée par $\kappa'$ : elle se montre comme pour les $G_\infty$-alg\`ebres (\cite{[AAC]}), donc $(1.2)=(2.2)+(3.2)$.\\

Dans tous les termes restants, l'\'el\'ement $X_0$ a \'et\'e coup\'e au moins une fois. Ces termes correspondent donc au cas où on coupe deux fois $X_0$ par $\kappa$ et au cas où on coupe une fois $X_0$ et une fois un des $X_t$ ($t>0$) par $\kappa$.\\

Commençons d'abord par le cas où on coupe deux fois $X_0$ par $\kappa$. Notons $X_0\mapsto U_0\otimes V_0\otimes W_0$ cette double c\'esure.

\begin{itemize}
\item[-] Dans $(1.1)$, on trouve les termes :
$$
\varepsilon_1U_0\bigotimes V_0\bigotimes W_0+\varepsilon_2U_0\bigotimes W_0\bigotimes V_0+\varepsilon_3W_0\bigotimes U_0\bigotimes
V_0+\varepsilon_4W_0\bigotimes V_0\bigotimes U_0.
$$

\item[-] Dans $(2.1)$, on trouve les termes :
$$
\nu_1U_0\bigotimes V_0\bigotimes W_0+\nu_2V_0\bigotimes U_0\bigotimes W_0+\nu_3V_0\bigotimes W_0\bigotimes U_0+\nu_4W_0\bigotimes V_0\bigotimes U_0.
$$

\item[-] Dans $(3.1)$, on trouve les termes :
$$
\rho_1U_0\bigotimes W_0\bigotimes V_0+\rho_2V_0\bigotimes W_0\bigotimes U_0+\rho_3V_0\bigotimes U_0\bigotimes W_0+\rho_4W_0\bigotimes U_0\bigotimes V_0.
$$
\end{itemize}

On vérifie que $\varepsilon_1=\nu_1$, $\varepsilon_2=\rho_1$, $\varepsilon_3=\rho_4$, $\varepsilon_4=\nu_4$, $\nu_2=-\rho_3$ et
$\nu_3=-\rho_2$. Expliquons par exemple l'\'egalit\'e $\varepsilon_1=\nu_1$.

Pour $\varepsilon_1$, on a effectu\'e les op\'erations ($Y_0=V_0\otimes W_0$) :
$$
X_0\stackrel{\kappa}{\longmapsto}(-1)^{u'_0}U_0\bigotimes Y_0\stackrel{-id\otimes\kappa}{\longmapsto}-(-1)^{u'_0}(-1)^{u'_0+v'_0}U_0\bigotimes V_0\bigotimes W_0.
$$
Donc $\varepsilon_1=(-1)^{v'_0+1}$.

Pour $\nu_1$, on a effectu\'e les op\'erations ($Y_0=U_0\otimes V_0$, $y'_0=u'_0+v'_0+1$) :
$$
X_0\stackrel{\kappa}{\longmapsto}(-1)^{u'_0+v'_0+1}Y_0\bigotimes W_0\stackrel{\kappa\otimes id}{\longmapsto}(-1)^{u'_0}(-1)^{u'_0+v'_0+1}U_0\bigotimes V_0\bigotimes W_0.
$$
Donc $\nu_1=(-1)^{v'_0+1}=\varepsilon_1$.\\

Etudions maintenant le cas où on coupe une fois $X_0$ et une fois un des $X_t$ ($t>0$) par $\kappa$.

Posons $X_0\mapsto U_0\otimes V_0$, et $X_t\mapsto U_t\otimes V_t$. Cherchons les termes correspondants. L'influence des $X_s$ ($s\neq t$) sur le signe provient uniquement de l'application de la r\`egle de Koszul. On peut donc les n\'egliger dans la suite du calcul.\\

\begin{itemize}
\item[-] Dans $(1.1)$, nécessairement $\kappa$ agit sur $X_0$ et $(id\otimes \kappa)$ sur $X_t$. On a donc les op\'erations
$$
X_0\otimes X_1\dots X_n\stackrel{\kappa}{\longmapsto}\pm U_0X_t\bigotimes V_0\pm U_0\bigotimes V_0X_t\pm V_0X_t\bigotimes U_0\pm V_0\bigotimes U_0X_t.
$$

On applique ensuite $-(id\otimes \kappa)$ sur $U_0\bigotimes V_0X_t$, les termes correspondants sont :
$$
\varepsilon_1U_0\bigotimes V_0U_t\bigotimes V_t+\varepsilon_2U_0\bigotimes V_0V_t\bigotimes U_t+\varepsilon_3U_0\bigotimes U_t\bigotimes
V_tV_0+\varepsilon_4U_0\bigotimes V_t\bigotimes U_tV_0.
$$

De m\^eme, on applique $-(id\otimes \kappa)$ sur $V_0\bigotimes U_0X_t$, les termes correspondants sont :
$$
\varepsilon_5V_0\bigotimes U_0U_t\bigotimes V_t+\varepsilon_6V_0\bigotimes U_0V_t\bigotimes U_t+\varepsilon_7V_0\bigotimes U_t\bigotimes
V_tU_0+\varepsilon_8V_0\bigotimes V_t\bigotimes U_tU_0.
$$

\item[-] On trouve de m\^eme 8 termes dans $(2.1)$ et 8 autres termes dans $(3.1)$, donc 16 termes dans $(2.1)+(3.1)$, 8 d'entre eux sont ceux de $(1.1)$, les 8 restants se simplifient deux \`a deux.\\

Par exemple le premier terme de $(1.1)$ est $\varepsilon_1U_0\bigotimes V_0U_t\bigotimes V_t$. Il appara\^it avec le signe :
$$
\varepsilon_1=-(-1)^{u'_0}(-1)^{u'_0+v'_0}(-1)^{u'_t}=(-1)^{u'_t+v'_0+1}.
$$

Ce terme n'appara\^it pas dans $(3.1)$ et il appara\^it une seule fois dans $(2.1)$, lorsque $\kappa$ coupe $X_t$ et $(\kappa\otimes id)$ coupe $X_0$. Il appara\^\i t avec le signe :
$$
(-1)^{u'_t}(-1)^{u'_0+v'_0+1}(-1)^{u'_0}= (-1)^{u'_t+v'_0+1}=\varepsilon_1.
$$

Enfin un exemple de simplification : dans $(2.1)$ on trouve un
seul terme de la forme $(U_0\otimes U_t)\bigotimes V_t\bigotimes
V_0$. Il appara\^\i t ainsi :
$$
X_0\otimes X_t\mapsto \pm (U_0\otimes X_t)\bigotimes V_0\mapsto \varepsilon_2 (U_0\otimes U_t)\bigotimes V_t\bigotimes V_0.
$$
Ce terme appara\^\i t une seule fois dans $(3.1)$, pr\'ecis\'ement dans $\tau_{23}(\kappa(X_0)\otimes\kappa'(X_1\dots X_n))$, ainsi :
$$
X_0\otimes X_t\mapsto \pm (X_0\otimes U_t)\bigotimes V_t\mapsto \pm (U_0\otimes U_t)\bigotimes V_0\bigotimes V_t\mapsto \varepsilon_3 (U_0\otimes U_t)\bigotimes V_t\bigotimes V_0.
$$
On obtient alors puisque $x'_0=u'_0+v'_0+1$ et $x'_t=u'_t+v'_t+1$,
$$\aligned
\varepsilon_2&=(-1)^{u'_0}(-1)^{v'_0x'_t}(-1)^{u'_t+u'_0}=(-1)^{u'_t+v'_0x'_t},\\
\varepsilon_3&=(-1)^{x'_0+u'_t}(-1)^{u'_0+v'_0u'_t}(-1)^{v'_0v'_t}=-\varepsilon_2.
\endaligned
$$
\end{itemize}

\end{proof}

En fait, les coproduits de Leibniz, $\kappa$ et permutatif, $\Delta$ ont des propri\'et\'es de compatibilit\'es, ce qui fait de $(\mathcal H[1]\otimes S(\mathcal H[1]),\Delta,\kappa)$ une bicog\`ebre au sens de Loday (\cite{[L2]}).\\

\begin{prop}

\

Les coproduits $\Delta$ et $\kappa$ vérifient les relations de compatibilité suivantes :
$$\aligned
(1):& ~(id \otimes
\kappa)\circ\Delta=\tau_{23} \circ(id\otimes \kappa)\circ\Delta,\\
(2):& ~(id\otimes \Delta)\circ\kappa=(\kappa\otimes id)\circ\Delta+\tau_{23}\circ(\kappa\otimes id)\circ\Delta,\\
(3):& ~(\Delta\otimes id)\circ\kappa=(id \otimes \kappa)\circ\Delta+\tau_{23}\circ(\kappa\otimes id)\circ\Delta.
\endaligned
$$
\vskip0.15cm

\end{prop}

\begin{proof}

\

\noindent
$(1)$ On rappelle que :
$$
\Delta (X_0\otimes X_1\dots X_n)=X_0\bigotimes X_1\dots X_n+(-1)^{x_0'}X_0\otimes\Delta'(X_1\dots X_n).
$$
où $\Delta'(X_1\dots X_n)$ est le coproduit défini sur la cogèbre cocommutative $S^{+}(\mathcal H[1])$. On a donc :
\begin{align*}
(id \otimes
\kappa)\circ\Delta (X_0\otimes X_1\dots X_n)=\Big(X_0\bigotimes\kappa'+(-1)^{x_0'}X_0\otimes\big(id \otimes
\kappa'\big)\circ\Delta'\Big)(X_1\dots X_n).\end{align*}
D'autre part,
$$\aligned
\tau_{23}\circ(id \otimes\kappa)\circ\Delta& (X_0\otimes X_1\dots X_n)=\\
&=\Big(X_0\bigotimes \tau_{23}\circ\kappa'+(-1)^{x_0'}X_0\otimes\big(id \otimes\tau_{23}\circ\kappa'\big)\circ\Delta'\Big)(X_1\dots X_n).
\endaligned
$$
Comme $\kappa'$ est cosymétrique, $\tau_{23}\circ\kappa'=\kappa'$. Ainsi,
$$
\tau_{23} \circ(id\otimes\kappa)\circ\Delta=(id \otimes\kappa)\circ\Delta.
$$

\noindent $(2)$ D'une part, on a :
$$\aligned
&(id\otimes \Delta)\circ\kappa(X_0\otimes X_1\dots X_n)=\\
&=(id\otimes \Delta)\Big(\sum_{\begin{smallmatrix}U_0\otimes V_0=X_0 \\ I\cup J=\{1,\dots,n\}\end{smallmatrix}} (-1)^{u_0'}\Big(\varepsilon_{x'}\left(\begin{smallmatrix}u_0~v_0~x_1\dots x_n\\ u_0~x_I~v_0~x_J\end{smallmatrix}\right) U_0\otimes X_I\bigotimes \mu V_0.X_J+\\
&\hskip 2cm+\varepsilon_{x'}\Big(\begin{smallmatrix}u_0~v_0~x_1\dots x_n\\ v_0~x_J~u_0~x_I\end{smallmatrix}\Big)\mu V_0\otimes X_J\bigotimes U_0 .
X_I\Big)+(-1)^{x_0'}X_0\otimes\kappa'( X_1\dots X_n)\Big)\\
&=\hskip-0.7cm\sum_{\begin{smallmatrix}U_0\otimes V_0=X_0 \\ I\cup J\cup K=\{1,\dots,n\};K\neq\emptyset\end{smallmatrix}}
\hskip -0.7cm(-1)^{u_0'}\Big\{\varepsilon_{x'}\Big(\begin{smallmatrix}u_0~v_0~x_1\dots x_n\\ u_0~x_I~v_0~x_{J\cup K}\end{smallmatrix}\Big) \Big(\varepsilon_{x'}\left(\begin{smallmatrix}v_0~x_{J\cup K}\\ v_0~x_J~x_ K\end{smallmatrix}\right) U_0\otimes X_I\bigotimes\mu V_0.X_J\bigotimes X_K+\\
&\hskip 4cm+\varepsilon_{x'}\Big(\begin{smallmatrix}v_0~x_{J\cup K}\\ x_K~v_0~x_J\end{smallmatrix}\Big) U_0\otimes X_I\bigotimes X_K\bigotimes \mu V_0.X_J\Big)+\\
&\hskip 4cm+\varepsilon_{x'}\Big(\begin{smallmatrix}
u_0~v_0~x_1\dots x_n\\ v_0~x_J~u_0~x_{I\cup K}\end{smallmatrix}\Big)\Big(\varepsilon_{x'}\left(\begin{smallmatrix}
u_0~x_{I\cup K}\\ u_0~x_I~x_ K\end{smallmatrix}\right)\mu V_0\otimes X_J\bigotimes U_0.X_I\bigotimes X_K+\\
&\hskip 4cm+\varepsilon_{x'}\Big(\begin{smallmatrix}
u_0~x_{I\cup K}\\ x_K~u_0~x_I\end{smallmatrix}\Big)\mu V_0\otimes X_J\bigotimes X_K\bigotimes U_0.X_I\Big)\Big\}+\\
&\hskip 4cm+(-1)^{x_0'}X_0\otimes(id\otimes \Delta')\circ\kappa'( X_1\dots X_n)\\
&=(1.1)+(1.2)+(1.3)+(1.4)+(1.5).
\endaligned
$$

D'autre part, on a :
$$
\aligned
&(\kappa\otimes id)\circ\Delta(X_0\otimes X_1\dots X_n)=\\
&=(\kappa\otimes id)\Big(X_0\bigotimes X_1\dots X_n+(-1)^{x_0'}X_0\otimes\Delta'(X_1\dots X_n)\Big)\\
&=\sum_{U_0\otimes V_0=X_0 }(-1)^{u_0'}\Big(\Big(U_0\bigotimes\mu V_0\bigotimes X_1 \dots X_n+\varepsilon_{x'}\Big(\begin{smallmatrix} u_0~v_0\\
v_0~u_0\end{smallmatrix}\Big)\mu V_0\bigotimes U_0\bigotimes X_1\dots X_n\Big)\\
&\hskip 3cm+\sum_{\begin{smallmatrix} I\cup J\cup K =\{1,\dots,n\}\\ I\cup K\neq\emptyset, J\neq\emptyset\end{smallmatrix}}\varepsilon_{x'}\left(\begin{smallmatrix} u_0~v_0~x_{I\cup J}~x_K\\ u_0~x_I~v_0~x_J~x_K\end{smallmatrix}\right) U_0\otimes X_I\bigotimes\mu V_0.X_J\bigotimes X_K+\\
&\hskip 3cm+\varepsilon_{x'}\Big(\begin{smallmatrix} u_0~v_0~x_{I\cup J}~x_K\\ v_0~x_J~u_0~x_I~x_K\end{smallmatrix}\Big)\mu V_0\otimes X_J\bigotimes U_0.X_I\bigotimes X_K\Big)+\\
&\hskip 3cm+X_0\otimes(\kappa'\otimes id)\circ\Delta'(X_1\dots X_n)\\
%\endaligned
%$$
%$$
%\aligned
&=\sum_{U_0\otimes V_0=X_0}(-1)^{u_0'}\sum_{\begin{smallmatrix}
I\cup J\cup K =\{1,\dots,n\}\\ K\neq\emptyset\end{smallmatrix}}
\varepsilon_{x'}\left(\begin{smallmatrix} u_0~v_0~x_1 \dots x_n\\
u_0~x_I~v_0~x_J~x_K\end{smallmatrix}\right)U_0\otimes
X_I\bigotimes
\mu V_0.X_J\bigotimes X_K+\\
&\hskip 3cm+\varepsilon_{x'}\left(\begin{smallmatrix} u_0~v_0~x_1 \dots~x_n\\ v_0~x_J~u_0~x_I~x_K\end{smallmatrix}\right)\mu V_0\otimes X_J\bigotimes
U_0.X_I\bigotimes X_K+\\
&\hskip 3cm+X_0\otimes(\kappa'\otimes id)\circ\Delta'(X_1\dots X_n),\\
\endaligned
$$
Ce qu'on note
$$
(\kappa\otimes id)\circ\Delta(X_0\otimes X_1\dots X_n)=(2.1)+(2.2)+(2.3).
$$
De m\^eme,
$$
\aligned
&\tau_{23}\circ(\kappa\otimes id)\circ\Delta(X_0\otimes X_1\dots X_n)=\\
&=\hskip-0.5cm\sum_{U_0\otimes V_0=X_0}(-1)^{u_0'}\hskip-0.5cm\sum_{\begin{smallmatrix} I\cup J\cup K =\{1,\dots,n\}\\
K\neq\emptyset\end{smallmatrix}}\hskip-0.5cm\varepsilon_{x'}\left(\begin{smallmatrix} u_0~v_0~x_1 \dots x_n\\
u_0~x_I~v_0~x_J~x_K\end{smallmatrix}\right)\varepsilon_{x'}\left(\begin{smallmatrix} v_0~x_J~x_K\\
x_K~v_0~x_J\end{smallmatrix}\right)U_0\otimes X_I\bigotimes X_K\bigotimes\mu V_0.X_J+\\
&\hskip 3cm+\varepsilon_{x'}\Big(\begin{smallmatrix} u_0~v_0~x_1 \dots~x_n\\
v_0~x_J~u_0~x_I~x_K\end{smallmatrix}\Big)\varepsilon_{x'}\left(\begin{smallmatrix} u_0~x_I~x_K\\
x_K~u_0~x_I\end{smallmatrix}\right)\mu V_0\otimes X_J\bigotimes X_K\bigotimes U_0.X_I+\\
&\hskip 3cm+X_0\otimes\tau_{23}\circ(\kappa'\otimes id)\circ\Delta'(X_1\dots X_n)\\
&=(3.1)+(3.2)+(3.3).
\endaligned
$$

On vérifie que $(1.5)=(2.3)+(3.3)$, gr\^ace \`a l'identit\'e de coLeibniz entre $\Delta'$ et $\kappa'$, \'etablie comme dans \cite{[AAC]}. De m\^eme, $(1.1)=(2.1)$, $(1.3)=(3.1)$, $(1.2)=(2.2)$ et $(1.4)=(3.2)$.\\

\noindent $(3)$ D'une part, on a
$$\aligned
&(\Delta\otimes id)\circ\kappa(X_0\otimes X_1\dots X_n)=\\
&=(\Delta\otimes id)\Big(\sum_{\begin{smallmatrix}U_0\otimes V_0=X_0 \\ I\cup J\cup K=\{1,\dots,n\}\end{smallmatrix}}
(-1)^{u_0'}\Big(\varepsilon_{x'}\left(\begin{smallmatrix}u_0~v_0~x_1\dots x_n\\ u_0~x_{I\cup K}~v_0~x_J\end{smallmatrix}\right) U_0\otimes X_{I\cup K}\bigotimes\mu V_0.X_J+\\
&\hskip 1cm+\varepsilon_{x'}\left(\begin{smallmatrix}u_0~v_0~x_1\dots x_n\\ v_0~x_{J\cup K}~u_0~x_I\end{smallmatrix}\right)\mu V_0\otimes X_{J\cup
K}\bigotimes U_0.X_I\Big)+(-1)^{x_0'}X_0\otimes\kappa'( X_1\dots X_n)\Big)\\
&=\sum_{\begin{smallmatrix}U_0\otimes V_0=X_0 \\ I\cup J\cup
K=\{1,\dots,n\} ;K\neq\emptyset\end{smallmatrix}}
(-1)^{u_0'}\Big\{\varepsilon_{x'}\left(\begin{smallmatrix}u_0~v_0~x_1\dots
x_n\\ u_0~x_I~x_K~v_0~x_J\end{smallmatrix}\right)
U_0\otimes X_I\bigotimes X_K\bigotimes\mu V_0.X_J+\\
&\hskip 1cm+\varepsilon_{x'}\left(\begin{smallmatrix}u_0~v_0~x_1\dots x_n\\ v_0~x_J~x_K~u_0~x_I\end{smallmatrix}\right)
\mu V_0\otimes X_J\bigotimes X_K\bigotimes U_0.X_I\Big\}+(-1)^{x_0'}X_0\bigotimes\kappa'( X_1\dots X_n)+\\
&\hskip 1cm+(-1)^{x_0'}X_0\otimes(\Delta'\otimes id)\circ\kappa'(X_1\dots X_n)\\
&=(1.1)+(1.2)+(1.3)+(1.4).
\endaligned
$$
D'autre part, on a
$$
\aligned
(id\otimes \kappa)&\circ\Delta(X_0\otimes X_1\dots X_n)=\\
&=(-1)^{x'_0}X_0\bigotimes\kappa'(X_1\dots X_n)+(-1)^{x'_0}X_0\otimes(id\otimes \kappa')\circ\Delta'(X_1\dots X_n)\\
&=(2.1)+(2.2).
\endaligned
$$
Et
$$
\aligned
\tau_{23}\circ(\kappa\otimes id)&\circ\Delta(X_0\otimes X_1\dots
X_n)=\\
&=\hskip-0.5cm\sum_{\begin{smallmatrix}U_0\otimes V_0=X_0 \\ I\cup J\cup K=\{1,\dots,n\} ;K\neq\emptyset\end{smallmatrix}}
 \hskip-0.5cm(-1)^{u_0'}\Big\{\varepsilon_{x'}\left(\begin{smallmatrix}u_0~v_0~x_1\dots x_n\\ u_0~x_I~x_K~v_0~x_J\end{smallmatrix}\right)
U_0\otimes X_I\bigotimes X_K\bigotimes\mu V_0.X_J+\\
&\hskip 2cm+\varepsilon_{x'}\left(\begin{smallmatrix}u_0~v_0~x_1\dots x_n\\ v_0~x_J~x_K~u_0~x_I\end{smallmatrix}\right)\mu V_0\otimes X_J\bigotimes X_K\bigotimes U_0.X_I\Big\}+\\
&\hskip 2cm+(-1)^{x'_0}X_0\otimes\Big(\tau_{23}\circ(\kappa'\otimes id)\circ\Delta'\Big)(X_1\dots X_n)\\
&=(3.1)+(3.2)+(3.3).
\endaligned
$$
L'identit\'e de coLeibniz donne $(1.4)=(2.2)+(3.3)$. On vérifie que
$$
(1.1)=(3.1),\quad (1.2)=(3.2),\quad (1.3)=(2.1).
$$
\end{proof}

On a ainsi muni l'espace $\mathcal H[1]\otimes S(\mathcal H[1])$ d'une structure de bicog\`ebre permutative et de Leibniz. On note cette bicog\`ebre $(\mathcal H[1]\otimes S(\mathcal H[1]),\Delta,\kappa)$.\\

%----------------------------------------------------------------

\section{Alg\`ebre pr\'e-Gerstenhaber \`a homotopie pr\`es}

\

On va montrer que les cod\'erivations $m$ et $R$ de $\Delta$,
obtenues \`a partir des lois de $\mathcal G$, sont aussi des
cod\'erivations du coproduit $\kappa$. Par cons\'equent $m+R$ est
une codiff\'erentielle \`a la fois pour $\Delta$ et pour $\kappa$.
On posera donc :

\begin{defn}

\

Une algèbre pré-Gerstenhaber \ à homotopie près est, pour le m\^eme opérateur $Q$, une cogèbre permutative codiff\'erentielle $(\mathcal{C},\Delta,Q)$ et une cogèbre de Leibniz codiff\'erentielle $(\mathcal{C},\kappa,Q)$ telle que les deux coproduits $\Delta$ et $\kappa$ satisfassent les relations de compatibilités suivantes :
$$\aligned
(id \otimes\kappa)\circ\Delta&=\tau_{23} \circ(id\otimes \kappa)\circ\Delta,\\
(id\otimes \Delta)\circ\kappa&=(\kappa\otimes id)\circ\Delta+\tau_{23}\circ(\kappa\otimes id)\circ\Delta,\\
(\Delta\otimes id)\circ\kappa&=(id \otimes \kappa)\circ\Delta+\tau_{23}\circ(\kappa\otimes id)\circ\Delta.
\endaligned
$$
\end{defn}

\begin{prop}

\

Soit $\mathcal G$ une alg\`ebre pr\'e-Gerstenhaber. Sur la bicog\`ebre $(\mathcal H[1]\otimes S(\mathcal H[1]),\Delta,\kappa)$, l'op\'erateur de degré $1$ et de carr\'e nul $Q=m+R$ est une cod\'erivation du coproduit $\kappa$.\\
\end{prop}

\begin{proof}

\

\noindent 1. Montrons d'abord que $m$ est une cod\'erivation de $\kappa$ : $(m\otimes id+ id\otimes m)\circ\kappa=-\kappa\circ m$.

On rappelle que :
$$\aligned
\kappa(X_0\otimes X_1\dots
X_n)&=\sum_{\begin{smallmatrix}U_0\otimes V_0=X_0\\ I\cup
J=\{1,\dots,n\}\end{smallmatrix}}(-1)^{u_0'}\times\Big(\varepsilon_{x'}\left(\begin{smallmatrix}
u_0v_0x_1\dots x_n\\ u_0~x_I~v_0~x_J\end{smallmatrix}\right)U_0\otimes X_I\bigotimes\mu V_0.X_J+\\
&+\varepsilon_{x'}\left(\begin{smallmatrix} u_0v_0x_1\dots x_n\\
v_0~x_J~u_0~x_I\end{smallmatrix}\right)\mu V_0\otimes
X_J\bigotimes U_0. X_I\Big)+(-1)^{x_0'}X_0\otimes\kappa'(X_1\dots
X_n)
\endaligned
$$
et
$$
m(X_0\otimes X_1\dots X_n)=D(X_0)\otimes X_1\dots X_n+(-1)^{x_0'}X_0\otimes m'(X_1\dots X_n),
$$
o\`u on a not\'e $m'$ la codiff\'erentielle $m$ de $\Delta'$ et $\kappa'$ dans $S^+(\mathcal H[1])$.

En développant, on trouve alors 10 termes dans $(m\otimes id+ id\otimes m)\circ\kappa$ et dans $-\kappa\circ m$.\\

- D'une part $(m\otimes id+ id\otimes m)\circ\kappa(X_0\otimes X_1\dots X_n)$ s'\'ecrit
$$\aligned
&\hskip-0.7cm\sum_{\begin{smallmatrix}U_0\otimes V_0=X_0\\ I\cup
J=\{1,\dots,n\}\end{smallmatrix}}\hskip-0.5cm(-1)^{u_0'}\varepsilon_{x'}\left(\begin{smallmatrix}
u_0v_0x_1\dots x_n\\
u_0~x_I~v_0~x_J\end{smallmatrix}\right)\Big(D(U_0)\otimes
X_I\bigotimes\mu V_0.X_J+(-1)^{u_0'}
U_0\otimes m'(X_I)\bigotimes \mu V_0.X_J\Big)\\
&\hskip
0.3cm+(-1)^{u_0'}\varepsilon_{x'}\left(\begin{smallmatrix}u_0v_0x_1\dots
x_n\\ v_0~x_J~u_0~x_I\end{smallmatrix}\right)\Big(
D(V_0)\otimes X_J\bigotimes U_0.X_I+(-1)^{v_0'}\mu V_0\otimes m'(X_J)\bigotimes U_0.X_I\Big)\\
&\hskip
0.3cm+(-1)^{u_0'+x_I'+u_0'}\varepsilon_{x'}\left(\begin{smallmatrix}u_0v_0x_1\dots
x_n\\ u_0~x_I~v_0~x_J\end{smallmatrix}\right)
U_0\otimes X_I\bigotimes D(\mu V_0).X_J+\\
&\hskip
0.3cm+(-1)^{u_0'+x_I'+u_0'+v'_0}\varepsilon_{x'}\left(\begin{smallmatrix}u_0v_0x_1\dots
x_n\\ u_0~x_I~v_0~x_J\end{smallmatrix}\right)
U_0\otimes X_I\bigotimes\mu V_0.m'(X_J)+\\
&\hskip
0.3cm+(-1)^{u_0'+x_J'+v_0'}\varepsilon_{x'}\left(\begin{smallmatrix}u_0v_0x_1\dots
x_n\\ v_0~x_J~u_0~x_I\end{smallmatrix}\right)
\mu V_0\otimes X_J\bigotimes D(U_0).X_I+\\
&\hskip
0.3cm+(-1)^{u_0'+x_J'+v_0'+u'_0}\varepsilon_{x'}\left(\begin{smallmatrix}u_0v_0x_1\dots
x_n\\ v_0~x_J~u_0~x_I\end{smallmatrix}\right)
\mu V_0\otimes X_J\bigotimes U_0.m'(X_I)+\\
&\hskip 0.3cm+(-1)^{x_0'}D(X_0)\otimes \kappa'(X_1\dots X_n)+\\
&\hskip 0.3cm+(-1)^{x_0'}X_0\bigotimes(m'\otimes id+ id\otimes m')\circ\kappa'(X_1\dots X_n)\\
&=(1.1)+\dots+(1.10).
\endaligned
$$

- De m\^eme, puisque $\mu D(V_0)=D(\mu V_0)$, et si $X_0=U_0\otimes V_0$,
$$
D(X_0)=D(U_0)\otimes V_0+(-1)^{u_0'}U_0\otimes D(V_0)=U'_0\otimes V_0+(-1)^{u'_0}U_0\otimes V'_0,
$$
et comme $m'(X_I.X_J)=m'(X_I).X_J+(-1)^{x'_I}X_I.m'(X_J)=X_I'.X_J\pm X_I.X'_J$, alors on d\'eveloppe $-\kappa\circ m(X_0\otimes X_1\dots X_n)$ en :

$$\aligned
&-\sum_{\begin{smallmatrix}U_0\otimes V_0=X_0\\ I\cup J=\{1,\dots,n\}\end{smallmatrix}}\Big((-1)^{u_0''}\varepsilon_{x'}\left(\begin{smallmatrix}
u_0'v_0x_1\dots x_n\\ u_0'~x_I~v_0~x_J\end{smallmatrix}\right)U_0'\otimes X_I\bigotimes\mu V_0.X_J+\\
&\hskip 1cm+(-1)^{u''_0}\varepsilon_{x'}\left(\begin{smallmatrix}u_0'v_0x_1\dots x_n\\ v_0~x_J~u_0'~x_I\end{smallmatrix}\right)
\mu V_0\otimes X_J\bigotimes U_0'.X_I+\\
&\hskip 1cm+(-1)^{u_0'+u'_0}\varepsilon_{x'}\left(\begin{smallmatrix}u_0v_0'x_1\dots x_n\\ u_0~x_I~v_0'~x_J\end{smallmatrix}\right)U_0\otimes X_I\bigotimes\mu V_0'.X_J+\\
&\hskip 1cm+(-1)^{u_0'+u'_0}\varepsilon_{x'}\left(\begin{smallmatrix}u_0v_0'x_1\dots x_n\\ v_0'~x_J~u_0~x_I\end{smallmatrix}\right)
\mu V_0'\otimes X_J\bigotimes U_0.X_I+\\
&\hskip 1cm+(-1)^{u_0'+x'_0}\varepsilon_{x'}\Big(\begin{smallmatrix}u_0v_0x_1\dots x_n\\ u_0~x_I'~v_0~x_J\end{smallmatrix}\Big)
U_0\otimes m'(X_I)\bigotimes\mu V_0.X_J+\\
&\hskip 1cm+(-1)^{u'_0+x_0'}\varepsilon_{x'}\left(\begin{smallmatrix}u_0v_0x_1\dots x_n\\ v_0~x_J~u_0~x_I'\end{smallmatrix}\right)\mu V_0\otimes X_J\bigotimes U_0.m'(X_I)+\\
&\hskip 1cm+(-1)^{u_0'+x_I'+x_0'}\varepsilon_{x'}\left(\begin{smallmatrix}
u_0v_0x_1\dots x_n\\ u_0~x_I~v_0~x_J'\end{smallmatrix}\right)U_0\otimes X_I\bigotimes \mu V_0.m'(X_J)+\\
&\hskip 1cm+(-1)^{u_0'+x_I'+x_0'}\varepsilon_{x'}\left(\begin{smallmatrix}u_0v_0x_1\dots
x_n\\ v_0~x_J'~u_0~x_I\end{smallmatrix}\right)
\mu V_0\otimes m'(X_J)\bigotimes U_0.X_I\Big)-\\
&\hskip 1cm-(-1)^{x_0'+1}D(X_0)\otimes \kappa'(X_1\dots X_n)-\\
&\hskip 1cm-(-1)^{x_0'+x'_0}X_0\bigotimes\kappa'\circ m'(X_1\dots X_n)\\
&=(2.1)+\dots+(2.10).
\endaligned
$$

On a imm\'ediatement $(1.10)=(2.10)$ car on sait que $m'$ est une cod\'erivation de $\kappa'$ (comme dans \cite{[AAC]}). On a aussi imm\'ediatement $(1.9)=(2.9)$. Les autres termes se simplifient deux \`a deux suivant le tableau :
$$\begin{array}{llll}
(1.1)=(2.1),~~&(1.2)=(2.5),~~&(1.3)=(2.4),~~&(1.4)=(2.8),\\
(1.5)=(2.3),&(1.6)=(2.7),&(1.7)=(2.2),&(1.8)=(2.6).
\end{array}
$$
V\'erifions par exemple l'\'egalit\'e des signes dans $(1.2)$ et $(2.5)$. Sh\'ematiquement :
$$
\begin{array}{ccccc}
X_0\otimes X_I.X_J&~\mapsto~&(U_0\otimes X_I)\bigotimes\mu V_0.X_J&~\mapsto~&(U_0\otimes m'(X_I))\bigotimes\mu V_0.X_J\\
\varepsilon_{x'}\left(\begin{smallmatrix}
x_0\dots x_n\\ x_0x_Ix_J
\end{smallmatrix}\right) &&(-1)^{u'_0}\varepsilon_{x'}\left(\begin{smallmatrix}v_0x_I\\ x_Iv_0\end{smallmatrix}\right)&&(-1)^{u'_0}
\end{array}
$$
Donc $\varepsilon_{(1.2)}$, le produit de ces trois signes, est $(-1)^{x'_Iv'_0}\varepsilon_{x'}\left(\begin{smallmatrix}x_0\dots x_n\\ x_0x_Ix_J\end{smallmatrix}\right)$.

Et
$$
\begin{array}{ccccc}
X_0\otimes X_I.X_J&~\mapsto~&X_0\otimes m'(X_I).X_J&~\mapsto~&(U_0\otimes m'(X_I))\bigotimes\mu V_0.X_J\\
-\varepsilon_{x'}\left(\begin{smallmatrix}x_0\dots x_n\\ x_0x_Ix_J\end{smallmatrix}\right) &&(-1)^{x'_0}&&(-1)^{u'_0}\varepsilon_{x'}\left(\begin{smallmatrix}v_0x'_I\\ x'_Iv_0\end{smallmatrix}\right)
\end{array}
$$
Donc
$$
\varepsilon_{(2.5)}=-(-1)^{v'_0+1}(-1)^{(x'_I+1)v'_0}\varepsilon_{x'}\left(\begin{smallmatrix}x_0\dots x_n\\ x_0x_Ix_J\end{smallmatrix}\right)=\varepsilon_{(1.2)}.
$$

\noindent
2. Montrons maintenant que $R$ est une cod\'erivation de $\kappa$ : $(R\otimes id+ id\otimes R)\circ\kappa=-\kappa\circ R$.

On rappelle que :
$$\aligned
R(X_0\otimes X_1\dots X_n)&=\sum_{i=1}^n\varepsilon_{x'}\Big(\begin{smallmatrix}x_1\dots x_n\\ x_i~x_1\dots\hat{i}\dots x_n\end{smallmatrix}\Big)R_2'(X_0, X_i) \otimes X_{1}\dots \widehat{_i}\dots X_{n}+\\
&\hskip 1cm(-1)^{x_0'}X_0\otimes\ell'(X_1\dots X_n),
\endaligned
$$
o\`u $\ell'$ est la cod\'erivation de la cog\`ebre $(S^+(\mathcal H[1]),\kappa')$ construite ci-dessus.\\

On identifie les termes apparaissant dans $(R\otimes id+ id\otimes R)\circ\kappa(X_0\otimes X_1\dots X_n)$ ainsi : si $\{1,\dots,n\}=I\cup J\cup\{s\}$, et $X_s$ est un des arguments de $R$:
$$\aligned
(1.1):~~ &\varepsilon_{(1.1)}R'_2(U_0,X_s)\otimes X_I\bigotimes\mu V_0.X_J,\\
(1.2):~~ &\varepsilon_{(1.2)}R'_2(\mu V_0,X_s)\otimes X_J\bigotimes U_0.X_I,\\
(1.3):~~ &\varepsilon_{(1.3)}R'_2(X_0,U_s)\otimes X_I \bigotimes\mu V_s.X_J,\\
(1.4):~~ &\varepsilon_{(1.4)}R'_2(X_0,\mu V_s)\otimes X_J \bigotimes U_s.X_I,\\
(1.5):~~ &\varepsilon_{(1.5)}U_0\otimes X_I \bigotimes \ell'_s(\mu V_0.X_s.X_J),\\
(1.6):~~ &\varepsilon_{(1.6)}\mu V_0\otimes X_J \bigotimes \ell'_s(U_0.X_s.X_I),\\
(1.7):~~ &\varepsilon_{(1.7)}U_0\otimes\ell'_s(X_s.X_I)\bigotimes\mu V_0.X_J,\\
(1.8):~~ &\varepsilon_{(1.8)}\mu V_0\otimes\ell'_s(X_s.X_J)\bigotimes U_0.X_I,\\
(1.9):~~ &\varepsilon_{(1.9)}R'_2(X_0,X_s)\otimes\kappa'(X_1\dots\hat{_s}\dots X_n),\\
(1.10):~~ &\varepsilon_{(1.10)}X_0\otimes (\ell'\otimes id+id\otimes\ell')\circ\kappa'(X_1 \dots X_n).
\endaligned
$$

De m\^eme, dans $-\kappa\circ R(X_0\otimes X_1\dots X_n)$, les termes qui apparaissent sont les suivants : où on a posé $R'_2(X_0,X_s)=U_{0s}\otimes V_{0s}$.
$$\aligned
(2.1):~~ &\varepsilon_{(2.1)}U_{0s}\otimes X_I\bigotimes\mu V_{0s}.X_J,\\
(2.2):~~ &\varepsilon_{(2.2)}\mu V_{0s}\otimes X_J\bigotimes U_{0s}.X_I,\\
(2.3):~~ &\varepsilon_{(2.3)}U_0\otimes X_I\bigotimes\mu V_0.\ell'_s(X_s.X_J),\\
(2.4):~~ &\varepsilon_{(2.4)}\mu V_0\otimes X_J\bigotimes U_0.\ell'_s(X_s.X_I),\\
(2.5):~~ &\varepsilon_{(2.5)}U_0\otimes \ell'_s(X_s.X_I)\bigotimes\mu V_0.X_J,\\
(2.6):~~ &\varepsilon_{(2.6)}\mu V_0\otimes \ell'_s(X_s.X_J)\bigotimes U_0.X_I,\\
(2.7):~~ &\varepsilon_{(2.7)}R'_2(X_0,X_s) \otimes \kappa'(X_1\dots\hat{_s}\dots X_n),\\
(2.8):~~ &\varepsilon_{(2.8)}X_0\otimes \kappa'\circ\ell'(X_1\dots X_n).
\endaligned
$$

Comme dans la preuve pr\'ec\'edente, les 4 derniers termes de ces deux listes se simplifient deux \`a deux :
$$
\begin{array}{llll}
(1.10)=(2.8),&(1.9)=(2.7),&(1.8)=(2.6),&(1.7)=(2.5).
\end{array}
$$

On d\'ecompose ensuite $(1.5)$ et $(1.6)$ en:
$$\aligned
(1.5)&=\varepsilon_{(1.5)}\left(U_0\otimes X_I \bigotimes \ell'_s(\mu V_0.X_s).X_J+(-1)^{v'_0}U_0\otimes X_I \bigotimes\mu V_0.\ell'_s(X_s.X_J)\right)\\
&=(1.5)_1+(1.5)_2,\\
(1.6)&=\varepsilon_{(1.6)}\left(\mu V_0\otimes X_J \bigotimes \ell'_s(U_0.X_s).X_I+(-1)^{u'_0}\mu V_0\otimes X_J \bigotimes U_0.\ell'_s(X_s.X_I)\right)\\
&=(1.6)_1+(1.6)_2.
\endaligned
$$

On v\'erifie que $(1.5)_2=(2.3)$ et $(1.6)_2=(2.4)$. Dans tous les termes restants, $X_I$ et $X_J$ sont de simples facteurs, on est en fait ramen\'e \`a prouver la relation sur $X_0\otimes X_s$.\\

Dans la suite de la preuve, on pose $s=1$.

Posons $X_0=\alpha_1\otimes\dots\otimes\alpha_p=\alpha_{[1,p]}$ et $X_1=\alpha_{p+1}\otimes\dots\otimes\alpha_{p+q}=\alpha_{[p+1,p+q]}$. Alors, on a
$$
\aligned
R_2'(X_0\otimes X_1)&=(-1)^{x_0'+\alpha_{p+1}\alpha_{[2,p]}}(\alpha_1\diamond \alpha_{p+1})\otimes sh(\alpha_{[2,p]},\alpha_{[p+2,p+q]})\\
&\hskip 0.5cm +\sum_{k=2}^p(-1)^{x_0'+\alpha_{p+1}\alpha_{[k+1,p]}}\alpha_{[1,k-1]}\otimes[\alpha_k,\alpha_{p+1}]\otimes sh(\alpha_{[k+1,p]},\alpha_{[p+2,p+q]}).
\endaligned
$$
En appliquant $-\kappa$, on obtient :

$$
\aligned
&-\kappa\circ R_2'(X_0\otimes X_1)=\\
&=\varepsilon_1(\alpha_1\diamond \alpha_{p+1})\bigotimes \mu\big(sh(\alpha_{[2,p]},\alpha_{[p+2,p+q]})\big)\\
&\hskip 0.5cm+\varepsilon_1'\mu\big(sh(\alpha_{[2,p]},\alpha_{[p+2,p+q]})\big)\bigotimes(\alpha_1\diamond \alpha_{p+1})\\
&\hskip 0.5cm+\sum_{\begin{smallmatrix}2\leq k\leq p+q\\ \sigma\in Sh_{p-1,q-1}\end{smallmatrix}}\varepsilon_2(\alpha_1\diamond \alpha_{p+1})\otimes
\alpha_{\sigma^{-1}([2,k])}\bigotimes\mu(\alpha_{\sigma^{-1}([k+1,p+q]\setminus\{p+1\})})\\
&\hskip 0.5cm+\sum_{\begin{smallmatrix}2\leq k\leq p+q\\ \sigma\in Sh_{p-1,q-1}\end{smallmatrix}}\varepsilon_2'\mu(\alpha_{\sigma^{-1}([k+1,p+q]\setminus\{p+1\})})\bigotimes(\alpha_1\diamond \alpha_{p+1})\otimes
\alpha_{\sigma^{-1}([2,k])}\\
&\hskip 0.5cm+\sum_{j<k\leq p}\varepsilon_3\alpha_{[1,j-1]}\bigotimes \mu\Big(\alpha_{[j,k-1]}\otimes[\alpha_k,\alpha_{p+1}]\otimes sh(\alpha_{[k+1,p]},\alpha_{[p+2,p+q]})\Big)\\
&\hskip 0.5cm+\sum_{j<k\leq p}\varepsilon_3'\mu\Big(\alpha_{[j,k-1]}\otimes[\alpha_k,\alpha_{p+1}]\otimes
sh(\alpha_{[k+1,p]},\alpha_{[p+2,p+q]})\Big)\bigotimes\alpha_{[1,j-1]}\\
\endaligned
$$
$$\aligned
&\hskip 0.5cm+\sum_{2\leq k\leq p}\varepsilon_4\alpha_{[1,k-1]}\bigotimes \mu\Big([\alpha_k,\alpha_{p+1}]\otimes
sh(\alpha_{[k+1,p]},\alpha_{[p+2,p+q]})\Big)\\
&\hskip 0.5cm +\sum_{2\leq k\leq p}\varepsilon_4'\mu\Big([\alpha_k,\alpha_{p+1}]\otimes
sh(\alpha_{[k+1,p]},\alpha_{[p+2,p+q]})\Big)\bigotimes\alpha_{[1,k-1]}\\
&\hskip 0.5cm+\hskip-0.5cm\sum_{\begin{smallmatrix}j> k\\ \sigma\in Sh_{j-k,p+q-j-1}\end{smallmatrix}}\hskip-0.8cm\varepsilon_5
\alpha_{[1,k-1]}\otimes[\alpha_k,\alpha_{p+1}]\otimes
\alpha_{\sigma^{-1}([k+1,j])}\bigotimes \mu(\alpha_{\sigma^{-1}([j+1,p+q])})\\
&\hskip 0.5cm+\hskip-0.5cm\sum_{\begin{smallmatrix}j> k\\ \sigma\in Sh_{j-k,p+q-j-1}\end{smallmatrix}}\hskip-0.8cm\varepsilon_5'
\mu(\alpha_{\sigma^{-1}([j+1,p+q])})\bigotimes\alpha_{[1,k-1]}\otimes[\alpha_k,\alpha_{p+1}]\otimes
\alpha_{\sigma^{-1}([k+1,j])}\\
&=(1)+(1)'+(2)+(2)'+(3)+(3)'+(4)+(4)'+(5)+(5)'.
\endaligned$$

Puisque $\mu\circ sh=0$, alors, $(1)=(1)'=0$.

Pour la m\^eme raison, dans $(2)$, et $(2)'$, les ensembles $\sigma^{-1}([a,b])$ sont inclus dans $[2,p]$ ou $[p+2,p+q]$, de m\^eme dans $(5)$, et $(5)'$, $\sigma^{-1}([j+1,p+q])$ est inclus dans $[k+1,p]$ ou $[p+2,p+q]$.

Donc
$$\aligned
(2)=(2)_0+(2)_1&=\sum_{\begin{smallmatrix}U_0\otimes V_0=X_0\setminus\{\alpha_{1}\}\\ \sigma\in Sh\end{smallmatrix}}
\pm(\alpha_1\diamond \alpha_{p+1})\otimes\sigma.(U_0\otimes X_1\setminus\{\alpha_{p+1}\})\bigotimes\mu(V_0)\\
&+\sum_{\begin{smallmatrix}U_1\otimes V_1=X_1\setminus \{\alpha_{p+1}\}\\ \sigma\in Sh\end{smallmatrix}}\pm(\alpha_1\diamond \alpha_{p+1})\otimes\sigma.(X_0\setminus\{\alpha_1\}\otimes U_1)\bigotimes\mu(V_1).
\endaligned
$$
Et de m\^eme pour $(2)'$.

Dans $(5)$, il reste :
$$\aligned
(5)=(5)_0+(5)_1=&\hskip-0.5cm\sum_{ \begin{smallmatrix}k, U^k_0\otimes V_0=X_0\setminus\alpha_{[1,k]}\\
\sigma\in Sh\end{smallmatrix}}\hskip-0.5cm
\pm\alpha_{[1,k-1]}\otimes[\alpha_k,\alpha_{p+1}]\otimes \sigma.(U^k_0\otimes X_1\setminus\{\alpha_{p+1}\})\bigotimes\mu(V_0)\\
&\hskip -0.5cm+\sum_{ \begin{smallmatrix}k, U_1\otimes V_1=X_1\setminus\{\alpha_{p+1}\}\\
\sigma\in Sh\end{smallmatrix}}\hskip-0.5cm \pm\alpha_{[1,k-1]}\otimes[\alpha_k,\alpha_{p+1}]\otimes \sigma.(X_0\setminus\alpha_{[1,k]}\otimes U_1)\bigotimes\mu(V_1).
\endaligned
$$
De m\^eme pour $(5)'$.

Enfin on remarque que :
$$\aligned
(4)+(3)=&\sum_{\begin{smallmatrix}U_0\otimes V_0=X_0\\ U_0=\alpha_{[1,k-1]}\end{smallmatrix}}\pm U_0\bigotimes R'_2(\mu V_0\otimes X_1)\\
&\hskip 0.5cm+\sum_{\begin{smallmatrix}U_0\otimes V_0=X_0\\ U_0=\alpha_{[1,k-1]}\end{smallmatrix}}\pm U_0\bigotimes\mu\Big([\alpha_k,\alpha_{p+1}]\otimes
sh\big(V_0\setminus\{\alpha_k\}, X_1\setminus\{\alpha_{p+1}\}\big)\Big).\\
\endaligned
$$
Donc 
$$
(4)+(3)=\sum_{U_0\otimes V_0=X_0}\pm U_0\bigotimes\ell'_2(\mu V_0.X_1).
$$

Et de m\^eme pour $(4)'+(3)'$.\\

On vérifie enfin les \'egalit\'es suivantes
$$\begin{array}{ccc}
(4)+(3)=(1.5)_1,&(4)'+(3)'=(1.2),&(2)_0+(5)_0=(1.1),\\
(2)'_0+(5)'_0=(1.6)_1,&(2)_1+(5)_1=(1.3),&(2)'_1+(5)'_1=(1.4).
\end{array}
$$

Ce qui ach\`eve la preuve.\\

\end{proof}

%--------------

On a ainsi montr\'e que $\left(\mathcal{H}[1]\otimes S(\mathcal{H}[1]),\kappa,Q\right)$ est une cog\`ebre de Leibniz codiff\'erentielle, c'est à dire, c'est une $Z_\infty$ algèbre.\\

\begin{thm}

\

Soit $(\mathcal G,\wedge,\diamond)$ une alg\`ebre pr\'e-Gerstenhaber, notons $\mathcal H=T^+(\mathcal G[1])$, $\Delta$, $\kappa$ les coproduits et $Q=m+R$, la cod\'erivation d\'efinis ci-dessus sur $\mathcal H[1]\otimes S^+(\mathcal H[1])$, alors
$$
\left(\mathcal{H}[1]\otimes S(\mathcal{H}[1]),\Delta,\kappa,Q\right)
$$
est une $preG_\infty$ alg\`ebre, appel\'ee la $preG_\infty$ alg\`ebre enveloppante de $(\mathcal G,\wedge,\diamond)$.\\
\end{thm}

%--------------

\end{document}